\documentclass[12pt]{iopart}   
\usepackage[utf8]{inputenc}
\usepackage[titletoc,title]{appendix}

\newcommand{\RR}{\mathbb{R}}

\usepackage{bm}
\newcommand{\grad}{{\bm{\nabla}}}
\renewcommand{\div}{\mathrm{div}}
\newcommand{\Lap}{\mathrm{\Delta}}

\newcommand{\Drm}{\mathrm{D}}

\newcommand{\T}{{\mathrm{T}}}

\newcommand{\Ocal}{\mathcal{O}}

\newcommand{\Abold}{\mathbf{A}}
\newcommand{\Bbold}{\mathbf{B}}
\newcommand{\Cbold}{\mathbf{C}}
\newcommand{\Dbold}{\mathbf{D}}
\newcommand{\Ibold}{\mathbf{I}}
\newcommand{\Jbold}{\mathbf{J}}
\newcommand{\Mbold}{\mathbf{M}}
\newcommand{\Nbold}{\mathbf{N}}
\newcommand{\Pbold}{\mathbf{P}}

\newcommand{\Sbold}{\mathbf{S}}
\newcommand{\Tbold}{\mathbf{T}}
\newcommand{\Wbold}{\mathbf{W}}

\newcommand{\dbold}{{\bm{d}}}
\newcommand{\ebold}{{\bm{e}}}
\newcommand{\fbold}{{\bm{f}}}

\newcommand{\nbold}{{\bm{n}}}

\newcommand{\ubold}{{\bm{u}}}

\newcommand{\xbold}{{\bm{x}}}

\newcommand{\zbold}{{\bm{z}}}

\newcommand{\lambdabold}{\bm{\lambda}}
\newcommand{\thetabold}{\bm{\theta}}

\newcommand{\mubold}{\bm{\mu}}

\newcommand{\blank}{\,{\cdot}\,}
\newcommand{\diag}{\mathrm{diag}}

\usepackage{tabularx,longtable,booktabs,wrapfig,multirow}

\newcommand*\diff{\mathop{}\!\mathrm{d}}

\usepackage{algorithmicx}
\usepackage[titlenumbered,ruled,vlined]{algorithm2e}
\usepackage{algpseudocode,float}
\algnewcommand\Input{\textbf{Input: }}
\algnewcommand\Package{\textbf{Package: }}
\algnewcommand\Parameters{\textbf{Parameters: }}
\algnewcommand\Output{\textbf{Output: }}

\usepackage[bookmarks=false,hyperfootnotes=false]{hyperref}
\hypersetup{
			colorlinks=true,
			linkcolor=red,
			anchorcolor=black,
			citecolor=green,
			urlcolor=blue
			}
\usepackage{amsthm}
\usepackage{amsfonts}
\usepackage{amssymb}
\expandafter\let\csname equation*\endcsname\relax
\expandafter\let\csname endequation*\endcsname\relax
\usepackage{amsmath}
\usepackage{mathtools}

\usepackage{graphicx}
\usepackage{color}
\usepackage{xcolor}
\usepackage{caption}
\usepackage{subcaption}

\theoremstyle{plain}
\newtheorem{theorem}{Theorem}[section]
\newtheorem{corollary}[theorem]{Corollary}

\newtheorem{lemma}[theorem]{Lemma}
\newtheorem{proposition}[theorem]{Proposition}

\theoremstyle{definition}
\newtheorem{definition}[theorem]{Definition}

\theoremstyle{remark}

\newtheorem{remark}[theorem]{Remark}

\makeatletter
\@namedef{ver@amsmath.sty}{}
\makeatother
\usepackage{amstext}
\usepackage[version=3]{mhchem}

\usepackage{footmisc}

\begin{document}

\title[Anisotropic osmosis filtering for shadow removal in images]{Anisotropic osmosis filtering for shadow removal \\ in images}

\author{Simone Parisotto$^1$, Luca Calatroni$^2$, Marco Caliari$^3$, Carola-Bibiane Sch\"onlieb$^4$ and Joachim Weickert$^5$}

\address{$^1$Cambridge Centre for Analysis, Wilberforce Road, Cambridge, CB3 0WA UK}
\address{$^2$CMAP, \'Ecole Polytechnique, 91128 Palaiseau Cedex, France}
\address{$^3$University of Verona, Strada Le Grazie 15, 37134, Verona, Italy}
\address{$^4$University of Cambridge, Wilberforce Road, Cambridge CB3 0WA, UK}
\address{$^5$Saarland University, Campus E1.7, 66041 Saarbr\"ucken, Germany}
\ead{sp751@cam.ac.uk, luca.calatroni@polytechnique.edu, marco.caliari@univr.it, cbs31@cam.ac.uk, weickert@mia.uni-saarland.de}
\vspace{10pt}
\begin{indented}
\item[]17 September 2018
\end{indented}






\begin{abstract}
We present an anisotropic extension of the isotropic osmosis model that has been introduced by Weickert et 
al.~\cite{weickert} for  visual computing applications, and we adapt it specifically to shadow removal
applications.
We show that in the integrable setting, linear anisotropic osmosis minimises an energy that involves a 
suitable quadratic form which models local directional structures. In our shadow removal applications we
estimate the local structure via a modified tensor voting approach \cite{MorBurWeiGarPui12} and use
this information within an anisotropic diffusion inpainting that resembles edge-enhancing anisotropic
diffusion inpainting \cite{WW06,GWWB08}.
Our numerical scheme combines the nonnegativity preserving stencil of Fehrenbach and Mirebeau
\cite{Mirebeau} with an exact time stepping based on highly accurate polynomial approximations of
the matrix exponential.
The resulting anisotropic model is tested on several synthetic and natural images corrupted by constant 
shadows. We show that it outperforms isotropic osmosis, since it does not suffer from blurring artefacts
at the shadow boundaries.

\end{abstract}

\section{Introduction}

The use of partial differential equations (PDEs) has a long tradition in mathematical image processing. In particular, PDEs based on transport and diffusion mechanisms have been considered to model several image reconstruction models suitable for image enhancement, denoising, deblurring, inpainting and segmentation. We refer the reader to the review \cite{Guichard2001} and the monographs \cite{aubert2006mathematical,ChanShen,Sa01,schoenlieb,weickert98} 
for further references.

\subsection{Anisotropic diffusion} 
Among these models, a very special place is occupied by diffusive PDEs encoding \emph{anisotropy}, i.e.\ favouring diffusion along some specific directions only. In \cite{weickert98} nonlinear diffusion PDEs with space-variant diffusion tensors are studied. 
For a regular image domain $\Omega\subset\RR^2$ and a stopping time $T>0$, given a degraded image $f\in L^\infty(\Omega;\RR)$ and two smoothing parameters $\rho,\sigma >0$, the anisotropic diffusion model in \cite{weickert98} looks for a solution $u$ in a suitable function space satisfying the following initial value problem:
\begin{equation} \label{eq:anis_diff}
\begin{dcases}
u_t = \div\Big( \Dbold(\Jbold_\rho(\grad u_\sigma))\grad u \Big) &\text{on } \Omega\times (0,T], \\
u(\xbold,0)=f(\xbold) & \text{on } \Omega,\\
\langle  \Dbold(\Jbold_{\rho}(\grad u_\sigma))\grad u), \nbold \rangle & \text{on }   \partial\Omega\times(0,T],
\end{dcases}
\end{equation}
where $\nbold$ is the outward normal unitary vector on $\partial\Omega$ and $\Dbold$ is a non-constant diffusion tensor satisfying suitable regularity conditions, with eigenvectors inherited from the so-called \emph{structure tensor} 
$\Jbold_\rho(\grad u_\sigma)$. 
It encodes local directional information of $u_\sigma$ (that is, the image $u$ convolved with a Gaussian kernel of standard deviation $\sigma$).  More precisely,
 its eigenvectors point in the directions of largest and smallest contrast averaged over a
 Gaussian smoothing scale $\rho$, and the corresponding eigenvalues measure this contrast.
 The diffusion tensor uses the same eigenvectors, and its eigenvalues are functions of the
 eigenvalues of the structure tensor. Depending on the application, different models such
 as edge-enhancing anisotropic diffusion or coherence-enhancing anisotropic diffusion
 have been proposed \cite{weickert98}
Edge-enhancing anisotropic diffusion has been adapted to inpainting problems in
\cite{WW06}. It is particularly useful for sparse inpainting problems encoutered e.g.\ 
in inpainting-based compression applications where it outperforms other PDE approaches
\cite{GWWB08,SPME14}.


Beyond anisotropic PDEs, non-smooth anisotropic regularisers for variational imaging models are also considered. 
In \cite{GL10,KonDon17,DongDTGV2017,ParMasSch18analysis,ParMasSch18applied}, for instance, directionality is used to define the anisotropic Total-Variation and Total-Generalised-Variation functionals.
This is classically done by considering a re-parametrised version of the gradient operator depending on the local orientation of the image, which allows to enforce diffusion along certain directions only. In terms of variational models, one replaces the squared norm $\grad^\top u \grad u$ by a quadratic form of type $\grad^\top u \Dbold  \grad u$. This has a very long tradition in image analysis \cite{NagelEnkelmann1986}.

The explicit dependence of these anisotropic models on local terms such as position and local directions of the image makes the analysis of these models more challenging \cite{weickert98,GL10,DongDTGV2017,ParMasSch18analysis}. Also from a numerical point of view, the design of suitable schemes enforcing anisotropy is a non-trivial task since it requires the use of appropriate stencils that perfectly adapt to the local image structure.
Many methods have been proposed; see e.g.\  the unifying framework in \cite{weickert2013b} and the references therein. While most stencils lead to $L^2$-stable schemes, only a few
of them allow to preserve nonnegativity and $L^\infty$ stabilty \cite{weickert98,MN01a,Mirebeau}. A rather
sophisticated representative among them is the stencil of Fehrenbach and Mirebeau \cite{Mirebeau} which relies on lattice basis reduction ideas.

\subsection{Osmosis filtering}  

In this work we consider a transport-diffusion PDE describing the physical phenomenon of \emph{osmosis} for imaging applications. Compared to standard plain diffusion models, the model considered therein considers an additional drift term, making the process not symmetric (see \cite{Hagenburg2012} for the physical interpretation).
For a regular domain $\Omega\subset\mathbb R^2$, a given vector field $\dbold: \Omega \to \RR^2$ and a given image $f\in L^\infty(\Omega;\RR)$, the isotropic osmosis model reads \cite{weickert}
\begin{equation}
\begin{dcases}
u_t = \Lap u - \div( \dbold u) &\text{on } \Omega\times(0,T],\\
u(\xbold,0) = f(\xbold) &\text{on } \Omega,\\
\langle\, \grad u - \dbold u, \nbold \,\rangle =0 &\text{on } \partial\Omega \times (0,T].
\end{dcases}
\label{eq: equation weickert}
\end{equation}
Differently from plain diffusion models, osmosis steady states are non-constant. In particular, if $\dbold$ is defined in terms of a given image $v>0$ as $\dbold:=\grad\log v$, convergence to a rescaled version of $v$ can be proven \cite{weickert}. This is called the integrable or
compatible case.
In \cite{weickert,vogel} several imaging applications based on \eqref{eq: equation weickert} or a slight modification thereof are studied. 
One of them is the shadow removal problem, see \cite[Section 4.2]{weickert}, as considered in this paper.

\paragraph{Shadow removal.} The problem of shadow removal from a given image $f:\Omega\to\RR_+$ consists in removing the shadow appearing in $f$ while preserving the image geometry and texture underneath. We will assume in the following \emph{constant} shadows, i.e.\ where image intensity values inside and outside the shadow region are in relation with each other up to an (unknown) multiplicative constant. 

This problem is of great interest in computer vision as it often represents a pre-processing step in several segmentation, tracking and face recognition tasks where shadows are removed to avoid false detections/artefacts in the subsequent image processing. 
We refer the reader to \cite{ShadowReview2013} for a review on the existing models for shadow removal in images. 


For our purposes, a mathematical formulation of the shadow removal problem can be obtained decomposing the image domain $\Omega$ as
\begin{equation}  \label{eq:decomposition Omega}
\Omega = \Omega_\text{out} \cup \Omega_\mathrm{sb} \cup \Omega_\text{in},
\end{equation}
where $\Omega_\text{out}, \Omega_\text{sb}$ and $\Omega_\text{in}$ are the unshadowed region, the shadow boundaries and the shadowed region of the image, respectively, see Figure \ref{fig:decomposition_omega} for an example.

\begin{figure}[tbhp]
\centering
\begin{subfigure}[t]{0.235\textwidth}\centering
\includegraphics[width=1\textwidth]{\detokenize{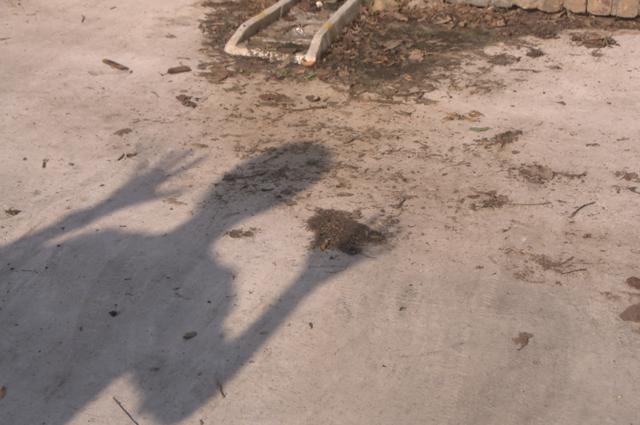}}
\caption{$\fbold$ on $\Omega$}
\label{fig:shadowed image}
\end{subfigure}
\,
\begin{subfigure}[t]{0.235\textwidth}\centering
\includegraphics[width=1\textwidth]{\detokenize{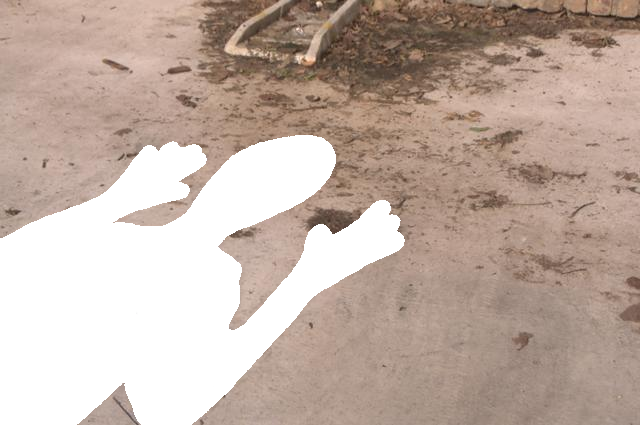}}
\caption{$\fbold$ on $\Omega_\text{out}$}
\end{subfigure}
\,
\begin{subfigure}[t]{0.235\textwidth}\centering
\includegraphics[width=1\textwidth]{\detokenize{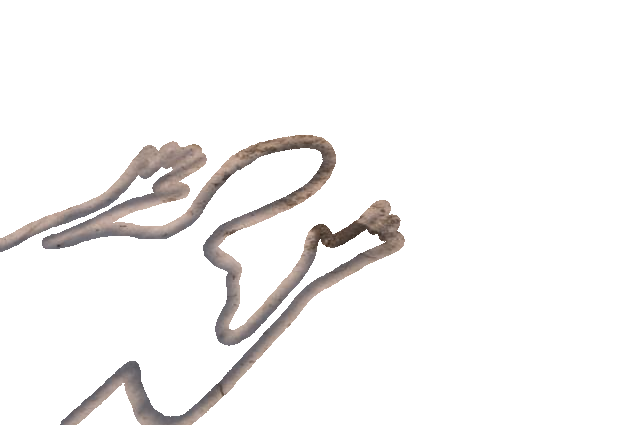}}
\caption{$\fbold$ on $\Omega_\text{sb}$}
\end{subfigure}
\,
\begin{subfigure}[t]{0.235\textwidth}\centering
\includegraphics[width=1\textwidth]{\detokenize{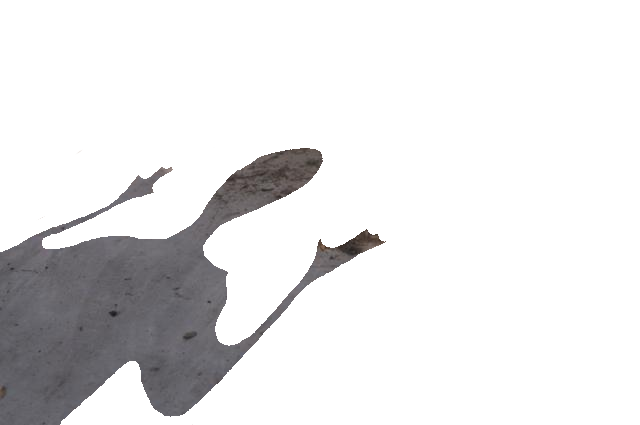}}
\caption{$\fbold$ on $\Omega_\text{in}$}
\end{subfigure}
\caption{Decomposition of $\Omega$ as in \eqref{eq:decomposition Omega} into (b) unshadowed region, (c) shadow boundaries and (d) shadowed region for a discrete shadowed image $\fbold$.}
\label{fig:decomposition_omega}
\end{figure}

Provided that a decomposition as in \eqref{eq:decomposition Omega} is given, which for real images may a challenging problem on its own regard \cite{ShadowReview2013}, the osmosis model \eqref{eq: equation weickert} can be easily adapted to solve the shadow removal problem by simply defining the vector field $\dbold$ in \eqref{eq: equation weickert} in terms of a shadowed image $f$ 
as
$\dbold:=\grad\log f$  on $\Omega_\text{in} \cup \Omega_\text{out}$ and $\dbold=0$ on  $\Omega_\text{sb}$. The continuous osmosis model adapted to shadow removal then reads

\begin{equation}
\begin{cases}
u_t = \Lap u - \div( \dbold u) &\text{on } \Omega_\text{in}\cup\Omega_{\text{out}}\times(0,T],\\
u_t = \Lap u  &\text{on } \Omega_\text{sb}\times(0,T],\\
u(\xbold,0) = f(\xbold) &\text{on } \Omega,\\
\langle\, \grad u - \dbold u, \nbold \,\rangle =0 &\text{on } \partial\Omega \times (0,T].
\end{cases}
\label{eq: equation weickert for shadow removal}
\end{equation}
The evolution on the shadow boundary $\Omega_\text{sb}$ can be interpreted as an inpainting step where information is propagated from $\Omega_\text{out}$ to $\Omega_\text{in}$ over $\Omega_\text{sb}$.  Due to the action of the Laplace operator on $\Omega_\text{sb}$, image structures in $\Omega_\text{in}\cup\Omega_{\text{out}}$ are isotropically diffused on $\Omega_\text{sb}$, resulting in a shadowless, but blurred inpainting result on $\Omega_\text{sb}$. To overcome this a post-processing inpainting step is commonly applied, as for instance in Figure \ref{fig: shadow removal with inpainting}. 

\begin{figure}[tbhp]
\centering
\begin{subfigure}[t]{0.33\textwidth}\centering
\includegraphics[width=1\textwidth]{\detokenize{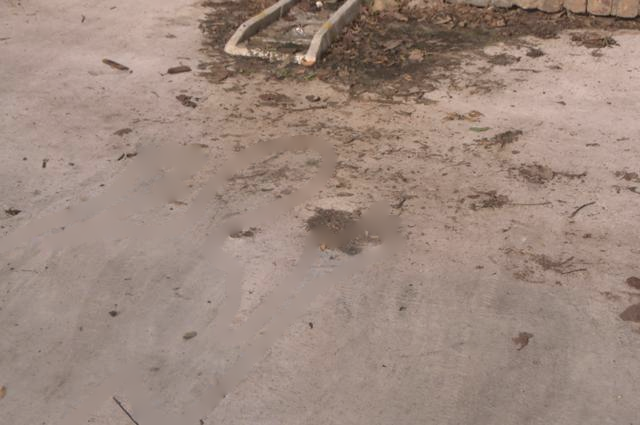}}
\captionsetup{justification=centering}
\caption{Isotropic osmosis \cite{weickert}}
\label{fig: isotropic filter}
\end{subfigure}
\hfill
\begin{subfigure}[t]{0.33\textwidth}\centering
\includegraphics[width=1\textwidth]{\detokenize{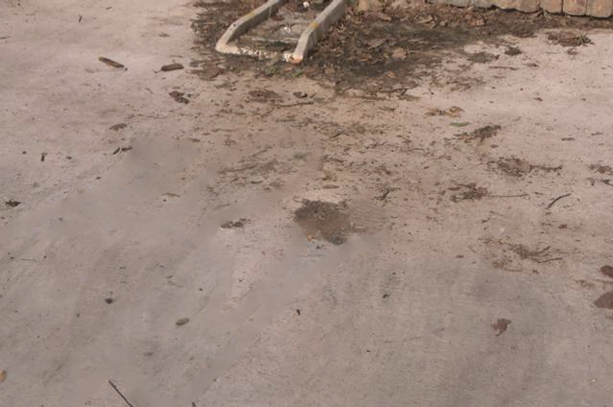}}
\captionsetup{justification=centering}
\caption{Post-processing inpainting step \cite{Arias2011}}
\label{fig: patch correction inpainting}
\end{subfigure}
\hfill
\begin{subfigure}[t]{0.32\textwidth}\centering
\vspace{-8.34em}
\includegraphics[width=0.863\textwidth,trim=8cm 5.5cm 8cm 6.85cm, clip=true]{\detokenize{./images/NCMIP/shadow_douglas.png}}
\\\vspace{0.3em}
\includegraphics[width=0.863\textwidth,trim=8cm 5.5cm 8cm 6.85cm, clip=true]{\detokenize{./images/NCMIP/inpainted_douglas_patch3x3.png}}
\captionsetup{justification=centering}
\caption{Zoom of Figure \ref{fig: isotropic filter} (top) and \ref{fig: patch correction inpainting} (bottom)}
\label{fig: patch correction inpainting zoom}
\end{subfigure}
\caption{Shadow removal for Figure \ref{fig:shadowed image} via \eqref{eq: equation weickert for shadow removal} in Figure \ref{fig: isotropic filter} and with post-processing inpainting correction in Figure \ref{fig: patch correction inpainting}
to remove the blurring artefacts due to Laplace inpainting on $\Omega_\text{sb}$ in \eqref{eq: equation weickert for shadow removal}.}
\label{fig: shadow removal with inpainting}
\end{figure}

Note that in natural images several acquisition and/or compression artefacts may render the automatic segmentation of the shadow boundary very challenging. On the other hand, its accurate manual selection may be very tedious. In many practical examples, a rough selection of $\Omega_\text{sb}$ is therefore performed manually by using a brush whose possibly large thickness may badly affect the result of the model \eqref{eq: equation weickert for shadow removal} (see Figure \ref{fig: thickness of shadow boundary}) due to the Laplace blurring artefacts  discussed above.

Vogel et al.\  \cite{vogel} have presented a discrete osmosis theory and have proven that
explicit and implicit finite difference discretisations satisfy its requirements. Different splitting schemes have been considered in \cite{Calatroni2017,ParCalDaf2018} and have been applied to imagery for cultural heritage conservation.

\subsection{Scope of the paper} 
In this paper we extend the isotropic osmosis model \eqref{eq: equation weickert} and its
shadow removal application \eqref{eq: equation weickert for shadow removal} to a model that features anisotropic diffusion in the flavour of \eqref{eq:anis_diff} and adapts concepts from anisotropic diffusion inpainting \cite{WW06,GWWB08}.
This will be implemented by incorporating local directionality depending on image orientations. We show  that with this modification we can improve the solution of the shadow removal problem, in particular overcoming blurring artefacts in the region around the shadow boundary, without additional post-processing steps. 
In the integrable case, the resulting drift-diffusion PDE can be derived as the gradient flow of a suitable energy depending on local gradient information. To estimate such local directionality, we adapt the tensor voting framework proposed in \cite{GidMed96}. For the numerical solution we combine a numerical time-stepping method based on exponential integrator techniques with the nonnegative space discretisation of Fehrenbach and Mirebeau \cite{Mirebeau}. Our model is validated on several synthetic and natural images affected by constant shadows.
Results show good light-balance properties and, compared to plain isotropic osmosis models, avoid the smoothing artefacts on the shadow boundary. An illustrative example of the performance of our model is reported in Figure \ref{fig:anisotropic_shadow}.

\begin{figure}[tbhp]
\centering
\begin{subfigure}[t]{0.32\textwidth}\centering
\includegraphics[width=0.8\textwidth]{\detokenize{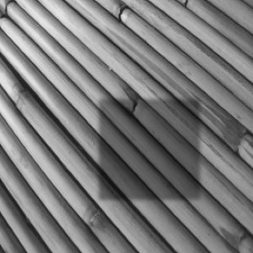}}
\caption{Shadowed image}
\label{fig: full bamboo}
\end{subfigure}
\,
\begin{subfigure}[t]{0.32\textwidth}\centering
\includegraphics[width=0.8\textwidth]{\detokenize{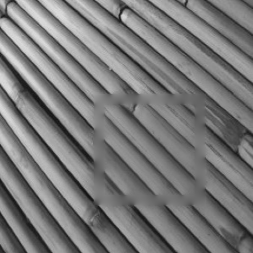}}
\caption{Isotropic osmosis \cite{weickert}}
\end{subfigure}
\,
\begin{subfigure}[t]{0.32\textwidth}\centering
\includegraphics[width=0.8\textwidth]{\detokenize{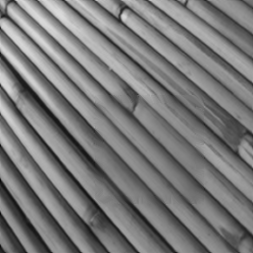}}
\caption{Proposed solution}
\end{subfigure}
\caption{Comparison of solutions obtained by solving the isotropic model considered in \cite{weickert,vogel} and our anisotropic one  to solve the shadow removal problem.}
\label{fig:anisotropic_shadow}
\end{figure}

\paragraph{Organisation of the paper.} 
In Section \ref{sec: anisotropic osmosis} we introduce the anisotropic osmosis model and study analytically its properties. Then, in Section \ref{sec:discrete theory} we study space and time discretisation schemes for the anisotropic model.
Finally, in Section \ref{sec: applications} we show the application of the anisotropic model to solve the shadow removal problem. 

\section{Anisotropic osmosis}\label{sec: anisotropic osmosis}

We present in this section a variation of the classical osmosis model \eqref{eq: equation weickert} encoding local directional information of the image in the diffusion term, propagating geometric structures dominantly along locally preferred directions. For this reason we call our model \emph{anisotropic} osmosis model in contrast to the model \eqref{eq: equation weickert} which we refer to as \emph{isotropic} osmosis model. 

In what follows we introduce the general form of our anisotropic osmosis model and state some properties of solutions that are inherited from the isotropic model. For specific choices of anisotropy we also show connections of the anisotropic osmosis model to anisotropic diffusion-based inpainting methods such as 
edge-enhancing anisotropic diffusion \cite{WW06}.
Out of these specific instances we derive our proposed anisotropic osmosis-inpainting model for shadow removal.

\subsection{Definitions and modelling}\label{sec: directional osmosis}

Let us define an anisotropic osmosis energy as  follows.
 
\begin{definition}[Anisotropic osmosis energy]  \label{eq: anisotropic energy}
Let $\Omega\subset\RR^2$, $u,v\in H^1(\Omega;\RR_+)$ be two positive images 
and let $\Wbold: \Omega \to \RR^2$ be a positive semi-definite symmetric matrix field.
We define the anisotropic osmosis energy of $u$ with respect to $\Wbold$ and the reference
image $v$ as
\begin{equation}
E(u) 
= 
\int_\Omega v(\xbold) \,  \grad^\top \!\left(\frac{u(\xbold)}{v(\xbold)}\right) \Wbold(\xbold) \, \grad \left(\frac{u(\xbold)}{v(\xbold)}\right) \diff \xbold.
\label{eq: energy with directions}
\end{equation}
\end{definition}

We will also use the following alternative notation for $E$:
\begin{equation}
E(u) 
= 
\int_\Omega v(\xbold)\; \left\|\grad \left(\frac{u(\xbold)}{v(\xbold)}\right)\right\|^2_{\Wbold}\diff \xbold
\label{eq: energy with directions equivalent}
\end{equation}
where $\|\ebold\|_{\Wbold}:=\sqrt{\langle \ebold,\, \Wbold  \ebold\rangle}$.

\begin{remark}[Isotropic case] \label{remark:isotropic}
If $\Wbold$ is the identity matrix, then \eqref{eq: energy with directions} corresponds to the isotropic osmosis energy considered in \cite{weickert}.
\end{remark}

Next we define an anisotropic osmosis evolution whose steady state minimises our
anisotropic osmosis energy.

\begin{proposition}\label{prop:ELanisomosis}
Let $v:\Omega\to\RR_+$ be a positive image, $\dbold\in \RR^2$ the vector field defined as $\dbold:=\grad\log v$ and $\Wbold: \Omega \to \RR^2$ be a positive semi-definite symmetric matrix field. Then, 
for a given positive image $f\in L^\infty(\Omega;\RR_+)$ the solution of the Euler-Lagrange equation of the functional $E$ defined in \eqref{eq: anisotropic energy} is the steady state of the 
anisotropic image osmosis model
\begin{equation}
\begin{cases}\label{eq: equation directional}
u_t = \div\left(\Wbold (\grad u - \dbold u) \right) &\text{on } \Omega\times(0,T],\\
u(\xbold,0) = f(\xbold) &\text{on } \Omega,\\
\langle\, \Wbold  \left( \grad u - \dbold u\right), \nbold \,\rangle =0 &\text{on } \partial\Omega \times (0,T].
\end{cases}
\end{equation}
\end{proposition}
\begin{proof}
For any test function $\varphi\in C^\infty_c(\Omega)$, we compute the optimality condition holding for any critical point $u$ of $E$.
We get:
\begin{align*}
\frac{\partial u}{\partial \tau} \left( E(u+\tau\varphi)\right)_{\vert \tau=0} 
&= 2\int_\Omega v~  \left\langle\Wbold\grad \left(\frac{u}{v}\right), \grad \left(\frac{\varphi}{v}\right)\right\rangle \diff \xbold
\\
&= 
-2\int_\Omega \div\left( v\Wbold\grad\left(\frac{u}{v}\right)\right)\frac{\varphi}{v} \diff \xbold 
\\
&= -2\int_\Omega\frac{1}{v}\div\left( v\Wbold\left(\frac{\grad u}{v} - \frac{u\grad v}{v^2}\right)\right)\varphi \diff \xbold = 0,
\end{align*}
where we have applied the divergence theorem and the Neumann boundary conditions in \eqref{eq: equation directional}. Due to the positivity of $v$ and since $\varphi$ is compactly supported in $\Omega$, we have that for any $\xbold\in\Omega$:
\begin{align*}
0 =\div\left( v\Wbold\left(\frac{\grad u}{v} - \frac{u\grad v}{v^2}\right)\right) 
=\div\left( \Wbold\left(\grad u - \frac{\grad v}{v} u\right)\right)  
\end{align*}
By definition of $\dbold= \frac{\grad v}{v}$, we note that the above corresponds to the following PDE:
\begin{align*}
\div\left(\Wbold\left(\grad u - \dbold u\right)\right) =0,
\end{align*}
which is the steady state of \eqref{eq: equation directional}.
\qedhere
\end{proof}

Similar to the isotropic osmosis PDE \eqref{eq: equation weickert}, the anisotropic model \eqref{eq: equation directional} enjoys some properties which makes it amenable for imaging applications.

\begin{theorem}  \label{theo: conservation anisotropic model}
The solution $u:\Omega\to\RR$ of the anisotropic osmosis model \eqref{eq: equation directional} satisfies the following properties:
\begin{enumerate}
\item \textbf{conservation of the average grey value}:
\[
\frac{1}{|\Omega|}\int_\Omega u(\xbold,t) \diff \xbold = \frac{1}{|\Omega|}\int_\Omega f(\xbold) \diff \xbold,\quad\text{for all } t>0;
\]
\item \textbf{preservation of non-negativity}: 
\[
u(\xbold,t)\geq 0,\quad \text{for all } \xbold\in\Omega\quad \text{and}\quad t>0;
\]
\item \textbf{non-constant steady states}: 
The steady state of \eqref{eq: equation directional} is given by 
\begin{equation}   \label{eq: steady state}
w(\xbold):=\frac{\mu_f}{\mu_v}v(\xbold).
\end{equation}
\end{enumerate}
\end{theorem}
\begin{proof}
We follow \cite{weickert} and prove statements (i)--(iii) in turn.
\begin{enumerate}
\item  Let $\mu_u(t):= \frac{1}{|\Omega|}\int_\Omega u(\xbold,t) \diff \xbold$ be the average grey value of the image $u$ at time $t\geq 0$. Applying the divergence theorem and the homogeneous Neumann boundary conditions in \eqref{eq: equation directional} we obtain:
\[
\begin{aligned}
\frac{d \mu_u}{d t} 
= 
\frac{1}{|\Omega|} \int_\Omega u_t \diff \xbold 
&= 
\frac{1}{|\Omega|}\int_\Omega \div \left( \Wbold  \left( \grad u - \dbold u \right)\right) \diff \xbold\\
&=
\int_{\partial \Omega} \langle\, \Wbold  \left( \grad u - \dbold u\right), \nbold  \,\rangle \diff S = 0,
\end{aligned}
\]
and the statement follows.
\item 
Assume that $T > 0$ is the smallest time such that $\min_{\xbold,t} u(\xbold, t) = 0$, and that
this minimum is attained in some inner point $\mubold\in\Omega$. Since we have $\grad u(\mubold, T) = u(\mubold,T)=0$, we deduce:
\begin{equation}
\begin{aligned}
u_t(\mubold,T) 
=& 
\div\left(\Wbold \grad u(\mubold, T)\right) - \div\left(\Wbold \dbold u(\mubold, T)\right)
\\
=& 
\Wbold \cdot \Drm^2 u(\mubold, T) + 
\left(\div(\Wbold )- \Wbold \dbold\right)\cdot\underbrace{\grad u(\mubold, T)}_{=0} 
-\div(\Wbold \dbold)\underbrace{u(\mubold,T)}_{=0}  
\\
=& 
\Wbold \cdot \Drm^2 u(\mubold, T),
\label{eq: point2 proof}
\end{aligned}
\end{equation}
that is in correspondence to the point $(\mubold,T)$, the anisotropic osmosis equation behaves like the diffusion equation
\begin{equation}  \label{eq: PDE diffusion}
u_t =  \Wbold \cdot \Drm^2 u,
\end{equation}
with non-constant positive semidefinite diffusivity matrix $\Wbold$. For such equations a generalisation of the weak minimum/maximum principle (which holds in its classical form only for positive definite,
i.e.\ elliptic, operators) holds true, see, e.g.\  \cite{ProtterWeinberger}. Therefore, for any $t\geq T$ the solution of the anisotropic model remains non-negative and since the solution has been further assumed to be strictly positive for all $t< T$, including $t=0$ due to the positivity of $f$, we have that it stays actually non-negative for any $t\geq 0$.

\item For any $c\in\RR$, the function $w:=cv$ endowed with the Neumann-type boundary conditions in \eqref{eq: equation directional} solves the steady state equation of the system anisotropic PDE, since there trivially holds:
$$
\div\left(\Wbold \left(\grad w - \frac{\grad v}{v}w\right)\right)=\div\left(\Wbold \left(c\grad v - \frac{\grad v}{v}cv\right)\right) = 0.
$$
Due to mass and non-negativity conservation, the process is then forced to a non-negative steady state solution $w$ of such a form. Furthermore, the constant $c\in\RR$ can be easily found by noticing
$$
c\mu_v=\frac{1}{|\Omega|}\int_\Omega cv(\xbold) \diff \xbold = \frac{1}{|\Omega|}\int_\Omega w(\xbold) \diff \xbold= \frac{1}{|\Omega|}\int_\Omega f(\xbold) \diff \xbold=\mu_f,
$$
whence $c=\mu_f/\mu_v$ which is well defined since $v$ is strictly positive in $\Omega$.\qedhere
\end{enumerate}
\end{proof}



\subsection{Anisotropic diffusion inpainting}

Anisotropic diffusion inpainting with a diffusion tensor has been introduced in \cite{WW06} and applied
successfully for inpainting-based compression \cite{GWWB08,SPME14}. It exploits the edge-enhancing anisotropic diffusion filter that has been proposed for image denoising \cite{weickert98}. In order to propagate structures from specified image regions into inpainting
regions, one uses the differential operator
$\div(\Dbold(\grad u_\sigma)\grad u)$, where $u_\sigma$ denotes the convolution of
$u$ with a Gaussian of standard deviation $\sigma$. The diffusion tensor $\Dbold$ has 
eigenvectors perpendicular and parallel to $\grad u_\sigma$. Its corresponding eigenvalues
are given by
\begin{equation}
 \mu_1(|\grad u_\sigma|)=1\quad\text{and}\quad
 \mu_2(|\grad u_\sigma|)=\frac{1}{\sqrt{1+|\grad u_\sigma|^2/\lambda^2}},
\end{equation}
with some contrast parameter $\lambda>0$.
Thus, the goal is to inpaint fully in the direction of an oriented structure, and to reduce the inpainting
perpendicular to a structures, if its contrast is large. Processes of this type can inpaint edge-like
structures even when the specified data are sparse and the gaps to be bridged are large \cite{SPME14}.
However, they are not well-suited to shadow removal problems, since the shadow boundaries 
create unphysical edges. Therefore, we will have to modify these ideas such that the local structure directions become more robust w.r.t.\  shadow boundaries. To this end, we will consider and modify more refined structure descriptors such as tensor voting. This will be done next.

\subsection{Computation of structure directions via tensor voting}
\label{sec: computation of directions}

In this section we present a work-flow which is locally insensitive to the light jump produced by the shadow. This will serve us to force anisotropy on $\Omega_{\text{sb}}$ along suitable directions.


A standard way to provide an estimate of the local structure orientation in an image $u$ consists in computing the eigenvector $\ebold_1$ associated to the leading eigenvalue $\lambda_1$ of the structure tensor $\Jbold_\rho(u)$ \cite{weickert98} associated to $u$. 
By fixing $\sigma,\rho>0$ to be the pre- and post-smoothing parameters, we recall that the structure tensor $\Jbold_\rho(u)$ of an image $u$ is defined as:
\[
\Jbold_\rho(u) := K_\rho \ast \Jbold_0(u),\quad\text{with}\quad \Jbold_0(u):= \grad u_\sigma \otimes \grad u_\sigma,
\]
where $u_\sigma := K_\sigma \ast u$ is a smoothed version of the image $u$ and $K_\sigma,K_\rho$ are Gaussian convolution kernels. 
Usually, $\sigma$ and $\rho$ are chosen so that $\sigma\ll \rho$, where $\sigma$ is associated to the noise scale and $\rho$ integrates the orientation information.  
Denoting by $\lambda_1\geq \lambda_2$ the eigenvalues of $\Jbold_\rho(u)$ and by $\ebold_1,\ebold_2$ the associated eigenvectors, a classification of the different structural image information in terms of the size of $\lambda_1$ and $\lambda_2$ can be made following, e.g.,  \cite{HarrisStephens88,Forstner86,KassWitkin85}. 
For any $\xbold\in\Omega$ we thus have:
\begin{itemize}
\item If $\lambda_1(\xbold)\approx \lambda_2(\xbold)\approx  0$, then $\xbold$ is likely to belong to a homogeneous region;
\item If $\lambda_1(\xbold)\gg \lambda_2(\xbold)\approx  0$, then $\xbold$ is likely to lie on an edge;
\item If $\lambda_1(\xbold)\approx \lambda_2(\xbold)\gg 0$, then $\xbold$ is likely to be a corner point.
\end{itemize}

\emph{Tensor Voting} has been originally introduced in \cite{GidMed96} for extracting curves in noisy images by means of the grouping of local features consistent in a neighbourhood of the measurements.
Such framework improves the robustness of structure tensor estimation in presence of noise and image artefacts \cite{MorBurWeiGarPui12}. 
Assuming that a generic 2-tensor $\Bbold$ in $\RR^2$ has the following matrix representation:
\begin{equation}
\Bbold =
\begin{pmatrix}
b_{11} & b_{12} \\
b_{21} & b_{22}
\end{pmatrix}
= 
\lambda_1 (\ebold_1 \otimes \ebold_1) + \lambda_2 (\ebold_2 \otimes \ebold_2)
\label{eq: matrix expansion}
\end{equation}
where $\lambda_{1}, \lambda_2$ are the eigenvalues associated to the eigenvectors $\ebold_{1}$ and $\ebold_{2}$, respectively,
we have that \eqref{eq: matrix expansion} can be equivalently rewritten as
\[
\Bbold = (\lambda_1-\lambda_2) (\ebold_1 \otimes \ebold_1) + \lambda_2 (\ebold_1 \otimes \ebold_1 + \ebold_2 \otimes \ebold_2).
\]
We can now distinguish the two following quantities:
\begin{itemize}
\item $(\lambda_1-\lambda_2)$ is called \emph{saliency} or \emph{stickness}.
It provides an estimation of the confidence on the direction $\ebold_1$ and it is also called \emph{orientation certainty} or \emph{anisotropy measure};
\item $\lambda_2$ is called \emph{ballness} and it measures the size of the minor-axis of the anisotropy ellipse. 
Since it is in some sense a measure of how the estimation of the main estimated direction is contradicted, it is often called \emph{orientation uncertainty} or \emph{junctionness}.
\end{itemize}
The tensor voting operation consists in adding, at each iteration, the contribution of neighbourhood tensors for each point in the domain, resulting in an enhanced tensor field due to the presence of saliency parameter.
In \cite{FraAlmRonFloRom06} the authors show that the complexity of the original approach may be time-consuming, even for small images. 
Thus, they propose an efficient computation of the tensor voting framework based on steerable filters theory, i.e.\ complex-valued convolutions. 

For the shadow removal application we are considering, the estimation of the structure direction $\ebold_1$  may be affected by the shadow edges, which do not correspond to the actual underlying image structures.
The presence of such edges make the use of either the structure tensor or the tensor voting very challenging for estimating the local directions.

To circumvent this problem, we propose a modification of the tensor voting framework by means of interpreting shadow edges as bias in the estimation of the structure.
From the given initial blurred image $u_\sigma$, we firstly compute the local \emph{orientation} $\theta$ of the gradient from the eigenvector $\ebold_2$.
Secondly, we compute the saliency of the given shadowed image. Finally we apply the tensor voting framework with a modified saliency and orientation on $\Omega_\text{sb}$ so as to mark the shadow boundaries as bias: there, the saliency is zeroed and the orientation is initialised at random.
Details on the algorithm and results are presented in Section \ref{sec: estimation local direction via tensor voting}.


\subsection{An anisotropic osmosis-inpainting model for shadow removal}



Let $f:\Omega\to\RR_+$ be a positive greyscale image with a constant shadow and let $\Omega$ be decomposed as 
in \eqref{eq:decomposition Omega}.
We propose the following structure-preserving osmosis model for solving the shadow removal problem:
\begin{equation}
\begin{dcases}
u_t = \div\left(\Wbold(\grad u - \dbold u \right)) &\text{on } \Omega\times (0,T],\\
u(\xbold,0) = f(\xbold) &\text{on }\Omega,\\
\langle\, \Wbold \left( \grad u - \dbold u\right),\nbold\,\rangle =0 &\text{on }\partial\Omega\times (0,T],
\end{dcases}
\label{eq: general shadow removal}
\end{equation}
Here we define the discontinuous vector field $\dbold$ and the discontinuous matrix field $\Wbold$ as
\begin{eqnarray}   
\label{eq:field_matrix}
\dbold(\xbold)
&=&
\begin{dcases}
\grad\log f,&\text{if } \xbold\in\Omega \setminus \Omega_\text{sb},\\
0,&\text{if } \xbold\in\Omega_\text{sb},
\end{dcases} 
\\
\Wbold(\xbold)
&=&
\begin{dcases}
\Ibold,\quad&\text{if } \xbold\in\Omega \setminus \Omega_\text{sb},\\
\epsilon (\ebold_1 \otimes \ebold_1) + 1 (\ebold_2 \otimes \ebold_2)
,\quad&\text{if } \xbold\in\Omega_\text{sb}. 
\label{eq:w_field}
\end{dcases}
\end{eqnarray}
where $\Ibold$ denotes the $2\times2$ identity matrix, and $\ebold_1$ and $\ebold_2$ are the directions from our modified tensor voting applied to the initial image $f$.


This means, an isotropic osmosis evolution is performed in $\Omega_\text{sb}^c$, while an anisotropic  inpainting is performed on $\Omega_\text{sb}$. In other words, classical osmosis balances image intensity in the shadowed region with respect to the unshadowed regions, while the nonlinear interpolation preserves structures and avoids blurring on the shadow boundary. The proposed model performs the osmosis and the inpainting step jointly and avoids any post-processing step.


\section{Space and time discretisation}  \label{sec:discrete theory}


 
We now discuss appropriate discretisations of the anisotropic osmosis model that are consistent with the continuous model \eqref{eq: equation directional} and preserve some of its properties as stated in Theorem \ref{theo: conservation anisotropic model}. To do so, we consider a discrete rectangular image domain with $M\times N$ pixels. Let $S:=MN$. The given positive initial image $\fbold$ is then defined as a vector in $\RR^S_{+}$. For a given grid step size $h>0$, we denote by $\ubold=(u_{i,j})_{i,j}$ the approximation of the function $u$ and with $u_{i,j}$ its approximated value in suitable discretisation nodes $((i-\frac{1}{2})h,(j-\frac{1}{2})h)$ with $i=1,\ldots,M$ and $j=1,\ldots,N$. 
Similarly, for $k\geq 0$ we denote by $u^k_{i,j}$ the value of $u_{i,j}$ at the time node $t_k=k \tau$, where $\tau$ is the time step size. 
Also, for $\xbold\in\Omega$, we denote by $\lambdabold_1,\lambdabold_2\in\RR_+^S$ the discretised eigenvalues $\lambda_1(\xbold)$ and $\lambda_2(\xbold)$, while by $\thetabold\in[0,2\pi)^S$ the discretised orientation $\theta(\xbold)$.

\subsection{Discrete osmosis theory}

A discrete theory for osmosis models has been established in \cite{vogel}. Since it is also applicable to the anisotropic 
setting, we report here the general result \cite[Proposition 1]{vogel} and list in the following some discrete solvers fulfilling its assumptions. 

\begin{theorem}[\cite{vogel}] \label{theo:vogel}
For a given  $\fbold\in\RR^S_+$, consider the fully-discretised problem:
\begin{equation}   \label{eq:semi_discrete}
\ubold^0 = \fbold, \qquad \ubold^{k+1} =\Pbold \ubold^k,\qquad k\geq 1,
\end{equation}
where the (unsymmetric) matrix $\Pbold\in \RR^{S\times S}$ is an irreducible, non-negative matrix with strictly positive diagonal entries and unitary column sum . 
Then the following properties hold true:
\begin{enumerate}
\item The evolution preserves positivity and the average grey value of $\fbold$;
\item The eigenvector of $\Pbold$ associated to eigenvalue $1$ is the unique steady state for $k\to\infty$.
\end{enumerate}
\end{theorem} 

As shown in \cite{vogel} for the isotropic osmosis model, standard explicit and implicit time finite difference
schemes fit this framework, the former being subject to time step size restrictions, the latter being unconditionally stable. Furthermore, both schemes converge to the space discretisation of the elliptic steady
state for every stable time step size. For the implicit scheme with a BiCGStab solver, Vogel et al.\ \cite{vogel} 
report speed-ups of
two orders of magnitude compared to the explicit method. In \cite{Calatroni2017} a Peaceman-Rachford splitting is shown to satisfy Theorem \ref{theo:vogel}  under a time step size restriction, while the additive operator splitting (AOS) considered in \cite{ParCalDaf2018} fulfils Theorem \ref{theo:vogel} for all time step sizes. Splitting schemes such as the AOS method, however, do not converge to the space-discrete elliptic steady state solution unless the time step size goes to zero. In practice, keeping this time splitting error under control imposes bounds on the time step size.

Note that in contrast to the theory for fully discrete diffusion filters \cite{weickert98}, the discrete osmosis theory in Theorem \ref{theo:vogel} does not require a symmetric matrix. Since it does also not involve any isotropy assumption, it is basically also applicable to anisotropic osmosis processes, if suitable space discretisations are employed. This, however, is not straightforward and will be discussed in the sequel.

\subsection{Space discretisation with the AD-LBR stencil} \label{sec:AD-LBR}
In what follows, we describe the discretisation of the weighting matrix $\Wbold$ and the differential operators $\div$ and $\grad$ for a grey-scale image of height $M$ and width $N$ unrolled as $\ubold\in\RR^{S}$, with $S=MN$, that defines the spatial discretisation matrix $\Abold$ for the semi-discrete osmosis problem
\begin{equation}\label{eq:ODE}
\begin{dcases}
\ubold'(t) = \Abold \ubold(t),&\text{for } t\in (0,T],\\
\ubold(0)= \fbold,
\end{dcases}
\end{equation}
where $T>0$ is a positive final time.

While diffusion processes are stable in many aspects, e.g.\ in terms of decreasing $L^2$ norms and decreasing $L^\infty$ norms, the stability of osmosis processes is restricted to essentially one
key property: the preservation of nonnegativity. Thus, any suitable space discretisation for osmosis filters should
guarantee that it is nonnegativity preserving. For the matrix $\Abold$ this implies that all off-diagonal elements
must be nonnegative. While this is easily satisfied for standard discretisations of isotropic processes \cite{vogel},
it becomes much  more challenging for anisotropic approaches: The drift term is fairly unproblematic 
and can be handled e.g.\ with classical upwind discretisations. However, most discretisations of the diffusion term
$\div (\Wbold\grad u)$ are only stable in the $L^2$ norm \cite{weickert2013b}. Thus, they cannot guarantee
preservation of nonnegativity. Nonnegativity-preserving discretisations can be found in \cite{weickert98,MN01a,Mirebeau}. In our paper, we use the stencil of Fehrenbach and Mirebeau \cite{Mirebeau}, since it is a fairly sophisticated nonnegativity-preserving method that has been reported to
give good results. This anisotropic diffusion discretisation relies on lattice basis reduction ideas and is called
AD-LBR.  Let us sketch its underlying ideas.

In \cite{Mirebeau}, the authors tackle the minimisation of an anisotropic energy 
which in our notation reads:
\[
E(u) = 
\int_\Omega \| \grad u(\xbold) \|_{\Wbold}^2 \diff \xbold,
\]
where $\|\ebold\|_\Wbold:=\sqrt{\langle \ebold,\, \Wbold \ebold\rangle}$, for any $\ebold\in\RR^d$ and $\Wbold$ is symmetric and positive definite.  
The idea is to introduce a discretisation $E_h$ of the energy above on the discretised domain $\Omega_h, h>0$ via a sum of weighted squared differences of $u\in L^2(\Omega_h;\RR)$, i.e.:
\begin{equation}
E_h(u) 
= 
h^{d-2} 
\sum_{\xbold\in\Omega_h} \sum_{\ebold\in V(\xbold)} \gamma_x(\ebold) |u(\xbold+h\ebold)-u(\xbold)|^2,
\label{eq: E_h mirebeau}
\end{equation}
where $V(\xbold)\subset\mathbb{Z}^d$ is the stencil and $\gamma_\xbold(\ebold)\geq 0$ are the associated weights. 
The key step is the linearisation of $u(\xbold+h\ebold)$ as $u(\xbold)+\langle \grad u(\xbold),\, h\ebold\rangle $, which shows that for each $\xbold\in\Omega_h$ and smooth $u$ one can write:
\[
h^d \| \grad u\|_{\Wbold}^2 
= 
h^{d-2} \sum_{\ebold\in V(\xbold)} \gamma_\xbold(\ebold) \langle \grad u(\xbold), h\ebold \rangle^2,
\]
which turns out to be equivalent to require the condition
$\Wbold = \sum_{\ebold\in V(\xbold)} \gamma_\xbold(\ebold) \ebold \ebold^\T
$.
via the following Lemma \cite[Lemma 1]{Mirebeau}.
\begin{lemma}
Let $\ebold_0,\ebold_1,\ebold_2\in\RR^2$ be such that $\ebold_0+\ebold_1+\ebold_2=0$ and $|\mathrm{det}(\ebold_1,\ebold_2)|=1$. Then for any symmetric positive definite matrix $\Wbold$, there holds:
\[
\Wbold = -\sum_{0\leq i \leq 2} \langle \ebold_{i+1}^\perp,\,\Wbold \ebold_{i+2}\rangle \ebold_i \ebold_i^\T,
\]
under the convention $\ebold_{3+i} := \ebold_i$ 
\end{lemma}
Actually, for any dimension $d\leq 3$ and any symmetric positive definite $d\times d$ matrix $\Mbold$, there always exists a family $(\ebold_i)_{i\in I}$ of vectors in $\RR^d$ such that $\langle \ebold_i, \Mbold \ebold_j \rangle\leq 0$ for any $i\neq j$. 
Such a family is called $\Mbold$-obtuse \cite{ConSlo}. 
Thus, by taking $(\ebold_0,\ebold_1,\ebold_2)$ as a $\Wbold$-obtuse superbase of $\mathbb{Z}^2$ (i.e.\ a basis of $\mathbb{Z}^2$ with $|\mathrm{det}(\ebold_0,\ebold_1,\ebold_2)|=1$ such that $\ebold_0+\ebold_1+\ebold_2=0$), a stencil $V(\xbold):=\{\ebold_0,\ebold_1,\ebold_2,-\ebold_0,-\ebold_1,\ebold_2\}$ can be used to write explicitly the coefficients $\gamma_\xbold$ in \eqref{eq: E_h mirebeau} for $0\leq i\leq 2$ as done in \cite[Equation (11)]{Mirebeau}:
\[
\gamma_\xbold(\pm \ebold_i) 
:= 
-\frac{1}{2} \langle \ebold_{i+1}^\perp,\,\Wbold  \ebold_{i+2}^\perp\rangle.
\]
The resulting stencil is shown to be independent of the choice of the superbase \cite[Lemma 11]{Mirebeau} and it is orientated along the preferred diffusion direction given by $\Wbold$. 
Also, the AD-LBR scheme is sparse, i.e.\ it has a limited support of 6 points for two-dimensional images.


In the AD-LBR discretisation, an important role is played by the anisotropic ratio $\kappa\in[1,\infty)$, which measures the geometrical shape of the ellipse associated to the eigen-decomposition of the anisotropic matrix $\Wbold$: the closer $\kappa$ is to 1, the more $\Wbold$ is similar to the identity matrix $\Ibold$.
The computation of the stencil has a logarithmic cost in the anisotropy ratio $\kappa$ of the diffusion tensor, making AD-LBR appealing for applications.
Also, by fixing a direction $\theta$ all over the domain $\Omega$ for $\Wbold$, the anisotropy ratio $\kappa$ can be related to the eigenvalues of $\Wbold$; see \cite[Equation 60]{Mirebeau}. For instance, if $\Wbold$ has eigenvalues $\varepsilon$ and $1$ with
$0<\varepsilon\ll 1$, then $\varepsilon=1/\kappa^2$.
%
%
For completeness, we report in Table \ref{tab: stencil mirebeau} the AD-LBR stencil for different choices of $\varepsilon$ and a fixed angle $\theta=2\pi/3$ that characterises the 
second eigenvector $\ebold_2=(\cos \theta, \sin \theta)^\top$ of $\Wbold$; cf.\ also \cite[Table 1,Table 2]{Mirebeau}.

\begin{table}[!htb]\footnotesize
\centering
\caption{AD-LBR stencil for the discretisation of $\div(\Wbold\grad (\blank))$: Different choices of $\varepsilon$ are presented with fixed angle $\theta=2\pi/3$ for the first eigenvector. In bold we denote the $(i,j)$ entry. As expected, we observe that in the case of strong anisotropy, the stencil becomes aligned in $\theta$ direction.}
\label{tab: stencil mirebeau}
\noindent\begin{tabularx}{\textwidth}{XXXX}
\toprule
\multicolumn{1}{c}{$\varepsilon=1$ ($\kappa=1$, $\Wbold=\Ibold$)}
& 
\multicolumn{1}{c}{$\varepsilon=0.5$ ($\kappa=\sqrt{2}$)}
& 
\multicolumn{1}{c}{$\varepsilon=0.1$ ($\kappa=\sqrt{10}$)}
&
\multicolumn{1}{c}{$\varepsilon=0.02$ ($\kappa=\sqrt{50}$)}
\tabularnewline  
\midrule
\[
\begin{aligned}
0.00 & &1.00 & &0.00\\
1.00 & &\mathbf{-4.00} & &1.00\\
0.00 & &1.00 & &0.00
\end{aligned}
\]
&
\[
\begin{aligned}
0.00    & &0.41          & &0.22\\
0.66 & &\mathbf{-2.57} & &0.66\\
0.22 & &0.41          & &0.00
\end{aligned}
\]
&
\[
\begin{aligned}
0.00 & &0.26 & &0.26\\
0.26 & &\mathbf{-1.16} & &0.26\\
0.26 & &0.26 & &0.00
\end{aligned}
\]
&
\[
\begin{aligned}
0.00 & &0.11 & &0.16\\
0.01 & &\mathbf{-0.55} & &0.01\\
0.16 & &0.11 & &0.00
\end{aligned}
\]
\tabularnewline  
\bottomrule
\end{tabularx}
\end{table}


\begin{remark}
\label{rem: AD-LBR column sum}
By construction, the AD-LBR space discretisation matrix $\Abold$ has 0 column sum and non-negative off-diagonal entries.
Moreover, our numerical experiments give strong evidence that $\Abold$ is also 
irreducible\footnote{We applied the Tarjan's algorithm finding the strongly connected components of a directed graph \cite{Tarjan1972}. Code is freely available at MATLAB central: \url{mathworks.com/matlabcentral/fileexchange/50707}.}. However, a formal proof of the irreducibility of $\Abold$ is left for future research. 
\end{remark}


\subsection{Exact time discretisation}
\label{sec: exponential integrators}
In this section we see how it is possible to solve the dynamical system~\eqref{eq:ODE}, through a highly accurate approximation of the exact solution
\begin{equation}
\ubold(T)=\exp(T\Abold)\fbold.
\end{equation}
First of all, we notice that it is not necessary to compute the large and dense matrix 
$\exp(T\Abold)$ explicitly. It is sufficient to compute only its
action on the initial solution $\fbold$. 
Polynomial methods (see, for instance, \cite{Saad92,Caliari16,AlMohy11}, approximate the action of the exponential by a polynomial of a certain degree applied to the initial vector. 
They do not require to solve a linear system of equations:
usually they scale the matrix and approximate the solution by
an iterative procedure like
\begin{equation*}
\ubold^{k+1}=p_{m_k}(\alpha_k\tau\Abold)\ubold^k,\quad k=0,1,\ldots,K-1,
\quad \ubold^0=\fbold
\end{equation*}
where $p_{m_k}$ is a polynomial of degree $m_k$ which approximates
the exponential function and
\begin{equation*}
\sum_{k=0}^{K-1}\alpha_k=1.
\end{equation*}
After the last iteration, $\ubold^K\approx \ubold(T)$. 
Such an iterative scheme is usually needed in order to achieve a sufficiently accurate result. For Krylov methods (like \cite{Saad92}), it is also necessary in order to keep the computational cost as low as possible, since the cost to produce $p_{m_k}$ requires $\Ocal(m_k^2)$ scalar products with vectors of the size of $\fbold$. 
Among the polynomial methods, the truncated Taylor series \cite{AlMohy11} and the interpolation in the Newton form at Leja points \cite{Caliari16} are able to bound the relative backward error.  
This means that they construct an approximation in $K$ iterations with time step size $\tau=T/K$:
\begin{equation*}
\ubold^{k+1}=p_m(\tau\Abold) \ubold^k,\quad k=0,1,\ldots,K-1
\end{equation*}
such that
\begin{equation*}
\left(p_m(\tau\Abold)\right)^K \fbold
=
\exp(T\Abold+\delta(T\Abold))\fbold,\quad
\text{with $\|\delta (T\Abold)\|\le \mathrm{tol}\cdot\|T \Abold\|$}.
\end{equation*}
The tolerance $\mathrm{tol}$ can be chosen as small as desidered. The typical value
for double precision arithmetic is $2^{-53}$ and in this sense the time
integration is said ``exact''. 
The polynomial $p_m(z)$ is either the truncated Taylor series of $e^z$ about
a point which depends on the spectrum of the matrix or the interpolation polynomial of $e^z$ at real Leja points on an interval related to the spectrum of the matrix.

The choice of the parameters $K$ and $m$ can be done by simply considering the 1-norm
of $T\Abold$, or estimates of $\|(T\Abold)^q\|_1^{1/q}$ for small
values of $q$, or, only for interpolation at Leja points, estimates on the $\epsilon$-pseudospectra of $T\Abold$. From the implementation point of
view, both algorithms simply require one matrix--vector product and one vector update at each degree elevation, and therefore the cost is $\Ocal(m)$. When the spectrum of $T\Abold$ has a skinny shape, either horizontal or vertical, usually interpolation at Leja points performs better (see~\cite{Caliari16}).

Since this method can be configured to be exact in time up to machine precision,
it is basically possible to reach the final time $T$ in a single time step of size
$\tau=T$. However, we prefer to use multiple time steps, since the steady state 
$T$ is in general not known a priori.
On the other hand, there is no restriction on the time step $\tau$ and it is
therefore possible to implement any desired strategy for steady state detection (based,
for instance, on the comparison of the solutions at two successive time steps).
In particular, also variable step size $\tau_k$ implementations are possible without
any restriction given by the stability or the computational cost.

%
The approximation of the action of the matrix exponential to a vector can
be used also in the so called \emph{exponential integrators}
(we refer to the survey paper~\cite{HochOst10})
when solving general (non-linear) stiff ordinary differential equations.

\subsection{Coverage through the discrete osmosis theory}  \label{sec:scale_space_properties}

A natural question in this context is whether the matrix $\Pbold:=\exp(\tau \Abold)$ satisfies the properties of Theorem \ref{theo:vogel} for suitable matrices $\Abold$. We answer this question with the following proposition, for which a simple lemma is required.
\begin{lemma}   \label{lemma:sum}
Let $\Cbold=(c_{i,j}), \Dbold=(d_{i,j})\in\RR^{S\times S}$ be matrices with column sums $c$ and $d$, respectively. Then the matrix $\Bbold:=\Cbold\Dbold$ has column sum $cd$.
\end{lemma}
\begin{proof}
For every $i,j=1,\ldots,N$, we write each element  $b_{i,j}$ in terms of the elements of $\Cbold$ and $\Dbold$. For every column $j$ we have:
\[
\sum_{i=1}^N b_{i,j} 
= 
\sum_{i=1}^N\sum_{k=1}^N c_{i,k}d_{k,j} 
= 
\sum_{k=1}^N\Bigl( d_{k,j} \sum_{i=1}^N c_{i,k}\Bigr) = \sum_{k=1}^N \Bigl( d_{k,j}\cdot c \Bigr) = c \sum_{k=1}^N d_{k,j} = c d.\qedhere
\]
\end{proof}

\begin{proposition}  \label{theo:scale-space_exp}
Let $\Abold$ be an irreducible matrix, with column sum $0$ and non-negative off-diagonal entries. 
Then the (non-symmetric) matrix $\Pbold:=\exp(\tau \Abold)$ is an irreducible positive matrix with column sum $1$.
\end{proposition}
\begin{proof}
We firstly show that $\Pbold$ has column sum $1$. To this end, we use the expansion
\begin{equation}  \label{eq:exponential:expansion}
\Pbold=\exp(\tau\Abold) = \sum_{k=0}^\infty \frac{\tau^k\Abold^k}{k!} = \Ibold + \sum_{k=1}^\infty \frac{\tau^k\Abold^k}{k!}.
\end{equation}

By hypothesis, $\Abold$ has column sum $0$.
Then the column sum of $\Abold^k$ is $0$ for every $k\geq 1$ by Lemma~\ref{lemma:sum}. Thus, the matrix $\Pbold=\exp(\tau\Abold)$ has column sum $1$.

In order to show that all elements of $\Pbold$ are positive, we rewrite $\Abold$ as 
\[
\Abold = \Dbold + \Nbold = \alpha \Ibold + (\Dbold - \alpha \Ibold + \Nbold)
\]
where $\Dbold=\diag(\Abold)$, $\Ibold$ is the identity matrix, $\Nbold=\Abold-\diag(\Abold)$ is a non-negative matrix and $\alpha:=\min_{i} a_{i,i} - \gamma$, for any $\gamma>0$. 
We have that $\Pbold$ can be expressed as the product
\[
\Pbold=\exp(\tau \Abold) = \exp(\tau \alpha \Ibold) \exp\left(\tau (\Dbold - \alpha \Ibold + \Nbold)\right) 
\]
where $\tau (\Dbold - \alpha \Ibold + \Nbold)$ is a non-negative matrix with positive diagonal and which is also irreducible. 
In fact, being that false, there would exist an extra-diagonal 
element in position $(i,j)$ such that $(\Dbold-\alpha\Ibold+\Nbold)_{i,j}^k=0$ for all $k>0$. 
But then, $(\Abold)_{ij}^k=\left(\alpha\Ibold+(\Dbold-\alpha\Ibold+\Nbold)\right)_{i,j}^k=0$ for all $k>0$, meaning that $\Abold$ is reducible, which is false by hypothesis. 
So, for any pair $(i,j)$ there exists a $k$ such that
$\left(\tau(\Dbold-\alpha\Ibold+\Nbold)\right)_{i,j}^k>0$. 
Since all the powers
of $\tau(\Dbold-\alpha\Ibold+\Nbold)$ appear in the series which defines
$\exp(\tau(\Dbold-\alpha\Ibold+\Nbold))$, we conclude that all its entries are
positive. Finally, $\exp(\tau \Abold)$ is obtained by scaling with positive
scalars the rows of $\exp(\tau(\Dbold-\alpha\Ibold+\Nbold))$. 
Therefore,
$\exp(\tau \Abold)$ is a positive (and thus irreducible) matrix.\qedhere
\end{proof}


Recalling Remark \ref{rem: AD-LBR column sum}, Proposition \ref{theo:scale-space_exp} shows that the AD-LBR space discretisation in combination with our ``exact''  polynomial 
time discretisation leads to a fully discrete anisotropic osmosis scheme that satisfies all assumptions of the discrete theory (subject to the missing formal irreducibility proof).


\section{Numerical results}\label{sec: applications}

In this section we present several numerical examples showing the application of the isotropic and anisotropic osmosis model to solve the shadow removal problem in synthetic and real-world images. We apply the anisotropic model \eqref{eq: general shadow removal} to images affected by almost constant shadows, in order to perform jointly the shadow removal and the inpainting procedure on the shadow edges.
Since in general the ground truth is not available for this problem, the quality of the reconstruction of the anisotropic model in comparison the isotropic approach is assessed by visual inspection.

\paragraph{On the thickness of the shadow boundary.} In the following numerical experiments we will assume for simplicity that a rough segmentation of the shadow boundary is provided beforehand. For natural real images, this may be a quite challenging task since, due to the possible presence of noise, blur and/or compression artefacts, such region may be not sharp and presents a blurred transition zone from the outside to the inside area of the shadow. As a consequence, standard segmentation methods based, for instance, on edge detection methods may fail. In fact, the task of shadow segmentation has been addressed on its own regard in previous works where brightness-based \cite{BabaSiggraph2003} or clustering \cite{ShadowReview2013} methods have been applied. However, in many practical situations, the shadow is often roughly detected manually by the user using a brush including pixels both from the inside and the outside of the shadowed area. Such region corresponds to $\Omega_{\text{sb}}$ in \eqref{eq:field_matrix}. 

In our experiments we have observed that if this region is chosen to be too small (i.e.\ smaller than the whole transition area between non-shadowed and shadowed area), then the shadow removal result is fairly poor. On the contrary, in general, selecting a thicker shadow boundary produces better results. Thick boundaries favour the use of the anisotropic model \eqref{eq: general shadow removal}-\eqref{eq:field_matrix} over the isotropic one \eqref{eq: equation weickert for shadow removal}: as we seen above, the action of the homogeneous diffusion smoothing on a large region $\Omega_{\text{sb}}$ is not able to preserve the underlying image structures. 

Figure \ref{fig: thickness of shadow boundary} illustrates these findings. We compare the result obtained applying the isotropic model \eqref{eq: equation weickert for shadow removal} on a real image where the thickness of the shadow boundary is chosen differently. We clearly observe that a thicker $\Omega_{\text{sb}}$ corresponds to a better removal of the shadow for the same large final time $T$.

\begin{figure}[tbhp]
\centering
\begin{subfigure}[t]{0.32\textwidth}\centering
\includegraphics[width=1\textwidth]{\detokenize{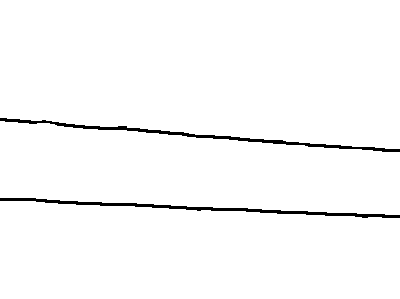}}
\caption{Mask of 3px}
\end{subfigure}
\,
\begin{subfigure}[t]{0.32\textwidth}\centering
\includegraphics[width=1\textwidth]{\detokenize{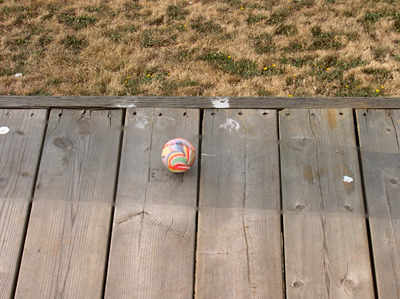}}
\caption{Result}
\end{subfigure}
\,
\begin{subfigure}[t]{0.32\textwidth}\centering
\includegraphics[width=1\textwidth,trim=3.5cm 2.5cm 5.86cm 4.5cm,clip=true]{\detokenize{./images/realshadow/66_classic.png}}
\caption{Result (zoom)}
\end{subfigure}
\\
\begin{subfigure}[t]{0.32\textwidth}\centering
\includegraphics[width=1\textwidth]{\detokenize{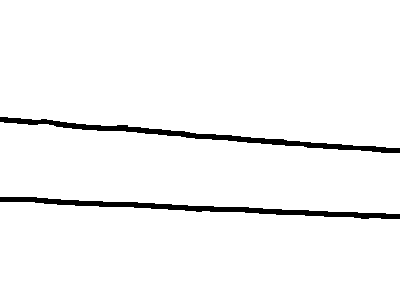}}
\caption{Mask of 5px}
\end{subfigure}
\,
\begin{subfigure}[t]{0.32\textwidth}\centering
\includegraphics[width=1\textwidth]{\detokenize{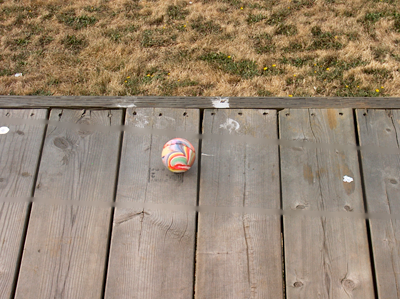}}
\caption{Result}
\end{subfigure}
\,
\begin{subfigure}[t]{0.32\textwidth}\centering
\includegraphics[width=1\textwidth,trim=3.5cm 2.5cm 5.86cm 4.5cm,clip=true]{\detokenize{./images/realshadow/6_classic.png}}
\caption{Result (zoom)}
\end{subfigure}
\caption{Comparison between different thicknesses of the shadow boundary for the solution of the isotropic model \eqref{eq: equation weickert for shadow removal}. Final time $T=100000$.}
\label{fig: thickness of shadow boundary}
\end{figure}

\medskip

We will now show the numerical results obtained when the anisotropic model \eqref{eq: general shadow removal}-\eqref{eq:field_matrix} was applied to a variety of both synthetic and natural images. 

\paragraph{Pseudocode.} 
Algorithm~\ref{alg: shadowremoval} describes the main steps required to solve the joint anisotropic osmosis problem after the estimation of local structures in the given image. For the time integration we used the \texttt{expleja} solver\footnote{Freely available at: \url{bitbucket.org/expleja/expleja}}, which is the companion software of \cite{Caliari16} for computing the action of the exponential matrix on a vector.

\begin{algorithm}[tbhp] 
\caption{Shadow Removal}
\SetAlgoLined\small
\label{alg: shadowremoval}
\SetKwProg{Fn}{Function}{:}{}
\SetKwFunction{computem}{compute\_matrix\_field}
\SetKwFunction{discretizematrix}{discretize\_matrix}
\SetKwFunction{shadowremoval}{shadow\_removal}
\SetKwFunction{expmv}{expleja}
\SetKwFunction{estimatedirection}{estimate\_direction}
\SetKwData{maxitert}{$T$}
\SetKwData{scales}{scales}
\SetKwData{H}{$\texttt{H}$}
\SetKwData{S}{$\texttt{S}$}
\SetKwData{V}{$\texttt{V}$}
\SetKwData{mask}{$\texttt{mask}$}
\SetKwInOut{Input}{Input}
\SetKwInOut{Package}{Package}
\SetKwData{colourchannel}{colour channel}
\BlankLine
\Input{a shadowed image $\fbold$ of dimension $M\times N\times C$ (with $C$ colour channels);\\ 
a \mask with value 1 on shadow boundary and 0 elsewhere;\\
a stack of $\scales = [s_1,\dots,s_S]$;\\
the parameters $\varepsilon$, $\sigma$, $\tau$ and $K$.}
\Package{\texttt{expjeja.m} from \url{bitbucket.org/expleja/expleja}}
\BlankLine
\Fn \shadowremoval {
\BlankLine
$\thetabold$ = \estimatedirection($\fbold$, \mask, \scales, $\sigma$)\tcp*{see Algorithm \ref{alg: tensor voting}}
$\Wbold$ = \computem{$\thetabold$, $\varepsilon$, \mask}
\tcp*{from Equation \eqref{eq:w_field}}
$\ubold^0=\fbold$\tcp*{initialisation}
\BlankLine
\For{$c=1,\dots,C$}{
\BlankLine
$\Abold_c$ = \discretizematrix{$\ubold^0(:,:,c)$, $\Wbold$, \mask}\tcp*{via AD-LBR stencil}
\BlankLine
\For{$k=0,\dots, K - 1$}{
\BlankLine
$\ubold^{k+1}(:,:,c)$ = \expmv{$\tau$, $\Abold_c$, $\ubold^k(:,:,c)$}\;
\BlankLine
}
\BlankLine
}
\BlankLine
}
\Return{$\ubold^{K}$.}
\end{algorithm}

\subsection{Synthetic examples}\label{sec: synthetic examples}
Firts we apply the anisotropic osmosis model to noise-free synthetic images, where the direction of the gradient $\zbold$ is known a priori.
The purpose of this synthetic experiment is to check if our approach is able to effectively remove constant shadows.
In Figure \ref{fig: comparison isotropic-anisotropic1} we show the results obtained for an image with parallel 
greyscale stripes with $\theta=65^\circ$ orientation and for colour concentric circles, whose $\theta$ is chosen to be as the angle drawn with tangent to the circumferences. Both images are corrupted with an almost constant shadow but a transition zone on the shadow borders.
We compare the anisotropic model described in Section \ref{sec:AD-LBR} with the isotropic osmosis model \eqref{eq: equation weickert for shadow removal}, which results in an homogenous diffusion inpainting at the shadow boundary. The time discretisation is performed as described in Section \ref{sec: exponential integrators}. 

In our visual comparison of both methods, we choose the final time $T=10000$ and the time-step $\tau=100$, and we proceeded using Algorithm \ref{alg: shadowremoval}. 
In this experiment we also provide a visual representation of the local orientation angle $\theta$ inside the
shadow boundary $\Omega_\text{sb}$. It becomes obvious that the anisotropic shadow removal
method shows clear advantages at the shadow boundaries, since it does not suffer from blurring artefacts.

\begin{figure}[tbhp]
\centering
\begin{subfigure}[t]{0.225\textwidth}\centering
\includegraphics[width=1\textwidth]{\detokenize{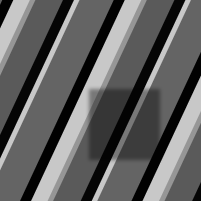}}
\caption*{Shadowed $\fbold$}
\end{subfigure}
\hfill
\begin{subfigure}[t]{0.225\textwidth}\centering
\includegraphics[width=1\textwidth]{\detokenize{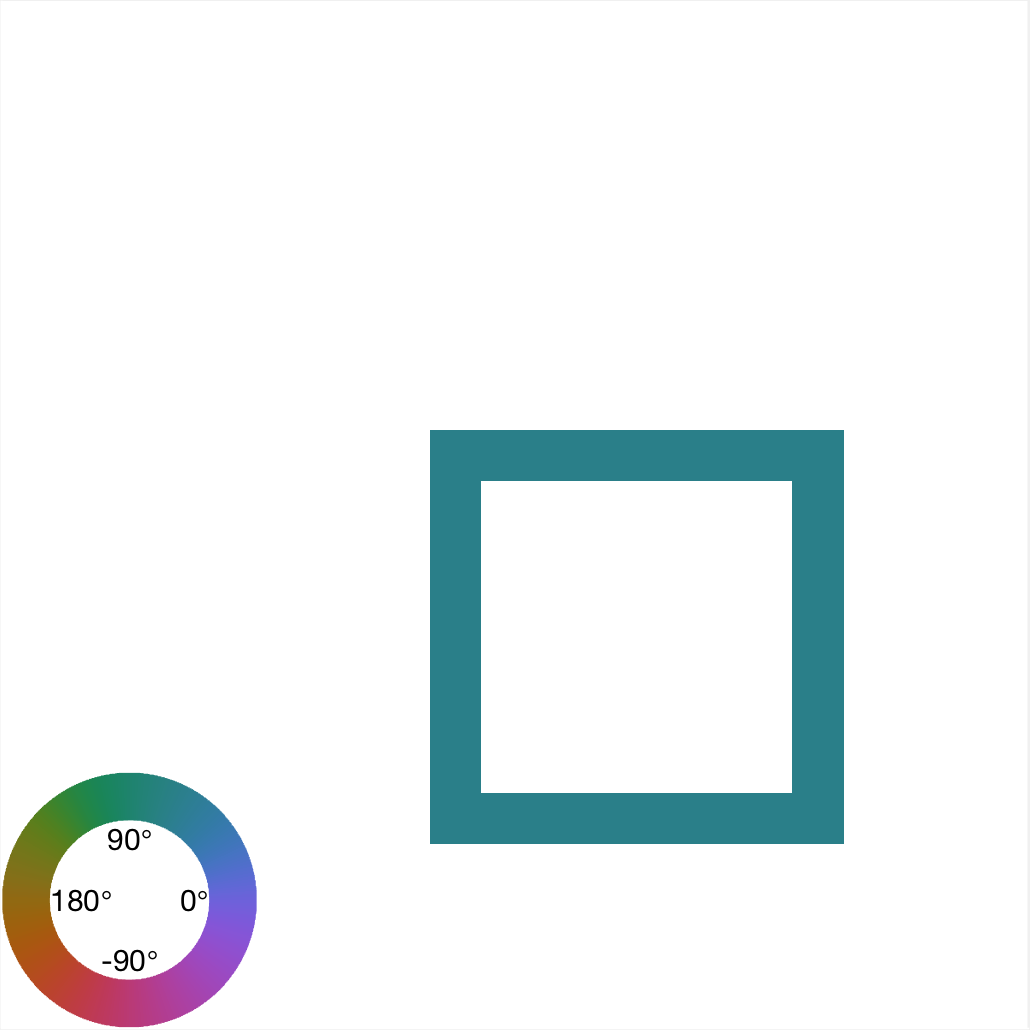}}
\captionsetup{justification=centering}
\caption*{Orientation angle $\thetabold$}
\end{subfigure}
\hfill
\begin{subfigure}[t]{0.225\textwidth}\centering
\includegraphics[width=1\textwidth]{\detokenize{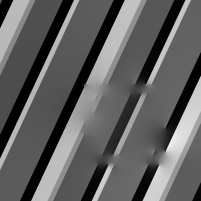}}
\captionsetup{justification=centering}
\caption*{Isotropic osmosis}
\end{subfigure}
\hfill
\begin{subfigure}[t]{0.225\textwidth}\centering
\includegraphics[width=1\textwidth]{\detokenize{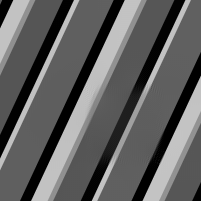}}
\captionsetup{justification=centering}
\caption*{Anisotropic osmosis}
\end{subfigure}
\\
\begin{subfigure}[t]{0.225\textwidth}\centering
\includegraphics[width=1\textwidth]{\detokenize{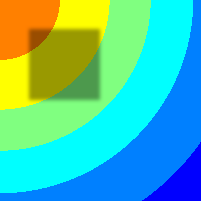}}
\caption*{Shadowed $\fbold$}
\end{subfigure}
\hfill
\begin{subfigure}[t]{0.225\textwidth}\centering
\includegraphics[width=1\textwidth]{\detokenize{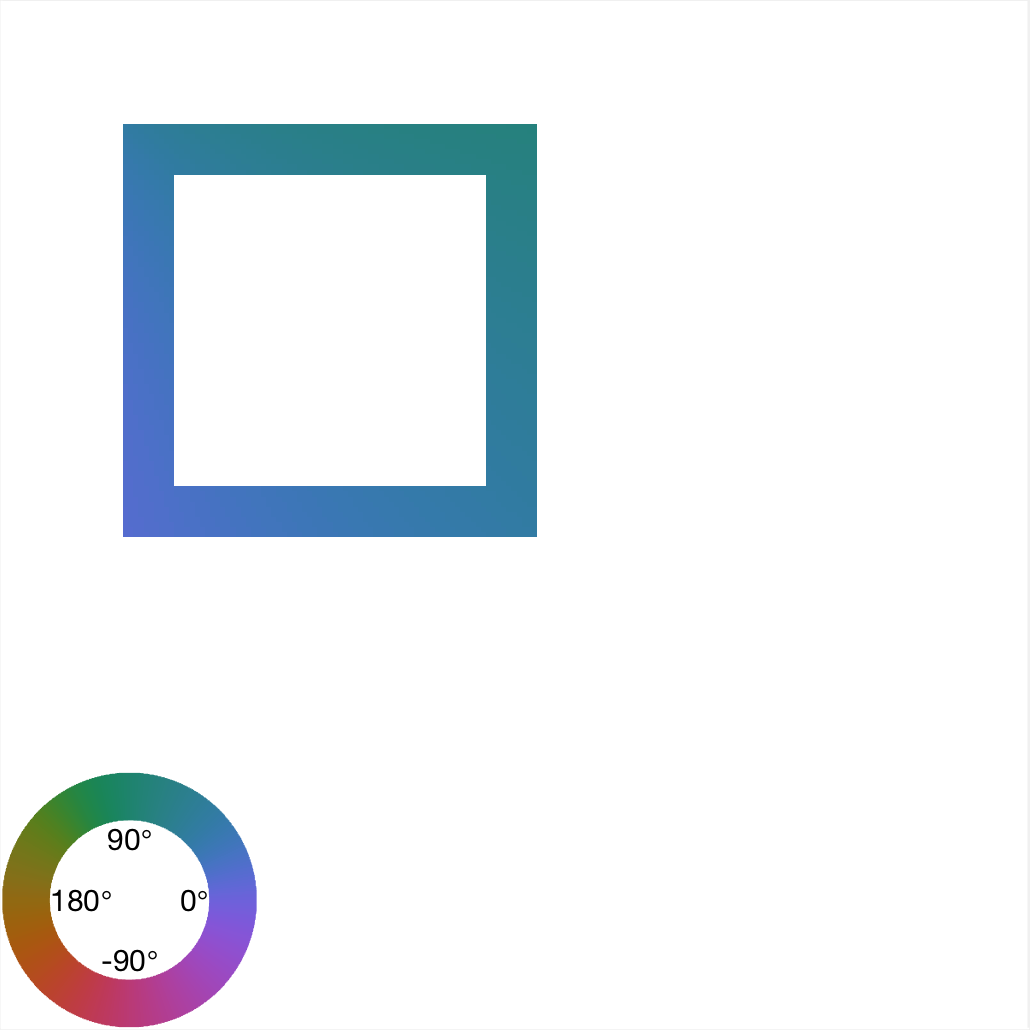}}
\captionsetup{justification=centering}
\caption*{Orientation angle $\thetabold$}
\end{subfigure}
\hfill
\begin{subfigure}[t]{0.225\textwidth}\centering
\includegraphics[width=1\textwidth]{\detokenize{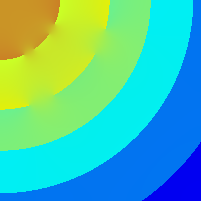}}
\captionsetup{justification=centering}
\caption*{Isotropic osmosis}
\end{subfigure}
\hfill
\begin{subfigure}[t]{0.225\textwidth}\centering
\includegraphics[width=1\textwidth]{\detokenize{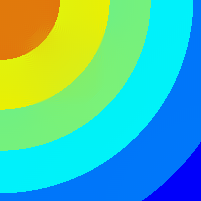}}
\captionsetup{justification=centering}
\caption*{Anisotropic osmosis}
\end{subfigure}
\caption{
Shadow removal via osmosis on synthetic images. 
Comparison between isotropic and aniotropic osmosis. 
Parameters: time step size $\tau=100$, final time $T=10000$, and smaller eigenvalue $\varepsilon=0.05$.}
\label{fig: comparison isotropic-anisotropic1}
\end{figure}

\subsection{Real-world examples}
In order to apply the anisotropic model to natural real-world images, the estimation of the discrete local orientation $\thetabold(\xbold)$ becomes crucial.
Therefore, we show in Section \ref{sec: estimation local direction via tensor voting} the results of the proposed approach discussed in Section \ref{sec: computation of directions}. 
Once the directions are estimated, we present in Section \ref{sec: real examples} the results of the anisotropic osmosis filter on real shadowed images.

\subsubsection{Estimation of the local orientation.}\label{sec: estimation local direction via tensor voting}

Following Section \ref{sec: computation of directions}, here we present the algorithm and the results for estimating the vector field that closes the interrupted lines onto the shadow boundary domain $\Omega_\text{sb}$, via the modified tensor voting framework.

For the tensor voting, we used the MATLAB implementation of \cite{FraAlmRonFloRom06} from the companion software\footnote{Freely available at MATLAB central: \url{mathworks.com/matlabcentral/fileexchange/47398}} of \cite{MagMenFre14}, a literature review on tensor voting.
Also, since tensor voting depends on local neighbourhoods, we use a multi-resolution strategy as is described in Algorithm \ref{alg: tensor voting}.

\begin{algorithm}[tbhp]
\caption{(Multi-scale) Eigen-directions via tensor voting in shadowed image}
\SetAlgoLined\footnotesize\small
\label{alg: tensor voting}
\SetKwProg{Fn}{Function}{:}{}
\SetKwFunction{computev}{compute\_vector\_field}
\SetKwFunction{discretizematrix}{discretize\_matrix}
\SetKwFunction{shadowremoval}{expmv}
\SetKwFunction{vote}{vote}
\SetKwFunction{encode}{encode}
\SetKwFunction{xytoij}{xy2ij}
\SetKwFunction{cos}{cos}
\SetKwFunction{sin}{sin}
\SetKwFunction{eigentensor}{eigen\_to\_tensor}
\SetKwFunction{tensoreigen}{tensor\_to\_eigen}
\SetKwFunction{estimatedirections}{estimate\_directions}
\SetKwData{maxitert}{$T$}
\SetKwData{H}{$\texttt{H}$}
\SetKwData{S}{$\texttt{S}$}
\SetKwData{V}{$\texttt{V}$}
\SetKwData{mask}{$\texttt{mask}$}
\SetKwInOut{Input}{Input}
\SetKwInOut{Package}{Package}
\SetKwData{colourchannel}{colour channel}
\SetKwData{salloc}{saliency\_loc}
\SetKwData{balloc}{ballness\_loc}
\SetKwData{oriloc}{orientation\_loc}
\SetKwData{saliency}{saliency}
\SetKwData{ballness}{ballness}
\SetKwData{orientation}{orientation}
\SetKwData{scales}{scales}
\BlankLine
\Input{a shadowed image $\fbold$ of dimension $M\times N\times C$ (with $C$ colour channels); 
\\
a \mask with value 1 on shadow boundary and 0 elsewhere;
\\
a stack of $\scales = [s_1,\dots,s_S]$;\\
the parameter $\sigma>0$.
}
\Package{\encode and \vote from \url{mathworks.com/matlabcentral/fileexchange/47398} modified to return \oriloc and \orientation in $(x,y)$ coordinates}
\BlankLine
\Fn \estimatedirections {
\BlankLine
TVF = zeros(size($\fbold$,1),size($\fbold$,2),2,2)\;
\BlankLine
\For{$c=1,\dots,C$}{
\BlankLine
[ \salloc,\,\balloc,\,\oriloc] = \encode($K_\sigma \ast \fbold(:,:,c)$)\;
\BlankLine
\tcp{zeroing and randomizing the data on the shadow edges $\Omega_\text{sb}$}
\salloc = \salloc.*\mask\;
\oriloc = \oriloc.*\mask + $2\pi$.* rand(size($\fbold(:,:,c)$)).*(1-\mask)\;  
\BlankLine
\ForEach{$k$ in \scales}{
	   \BlankLine
       [ \saliency,\,\ballness,\,\orientation] = \vote( \salloc,\,\oriloc,\, $s_k$ )\;
	   $\lambdabold_1$ = \saliency+\ballness\;
	   $\lambdabold_2$ = \ballness\;	
       $\ebold_1$ = ( cos(\orientation), sin(\orientation))\;
	   $\ebold_2$ = (-sin(\orientation), cos(\orientation))\;
        
       TVF = TVF + \eigentensor( $\ebold_1$, $\ebold_2$, $\lambdabold_1$, $\lambdabold_2$ )\;
\BlankLine
}
\BlankLine
}
\BlankLine
[ $\ebold_1$, $\ebold_2$, $\lambdabold_1$, $\lambdabold_2$ ] = \tensoreigen( TVF )\;
$\thetabold$ = \xytoij($\ebold_2$)\tcp*{return the local orientation in $(i,j)$ coordinates}
\BlankLine
}
\Return{$\thetabold$}
\end{algorithm}

In Figure \ref{fig: STVvsTVF} we compare the local direction estimation by means of the structure tensor and the tensor voting approach applied on the shadowed image in Figure \ref{fig: STVvsTVF1}. We plot the magnitude of the leading eigenvalue of the structure tensor in Figure \ref{fig: STVvsTVF2} and of the one computed for the tensor voting Algorithm \ref{alg: tensor voting} in Figure  \ref{fig: STVvsTVF3}. These images clearly show that the estimation via tensor voting is less sensitive to the false edges introduced by the shadow boundaries. 
This is reflected in the plot of the leading directions, too: The directions computed via the structure tensor in Figure \ref{fig: STVvsTVF4} are visibly affected by the light jumps while the ones computed with tensor voting in Figure \ref{fig: STVvsTVF5} can still  connect the structures from outside to inside the shadow with no additional edges.

\begin{figure}[tbhp]
\begin{subfigure}[t]{0.193\textwidth}
\includegraphics[width=1\textwidth]{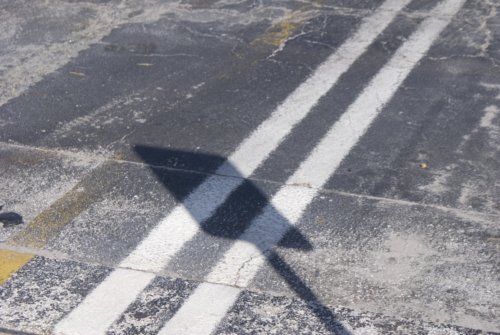}
\caption{Input $\fbold$}
\label{fig: STVvsTVF1}
\end{subfigure}
\hfill
\begin{subfigure}[t]{0.193\textwidth}
\includegraphics[width=1\textwidth]{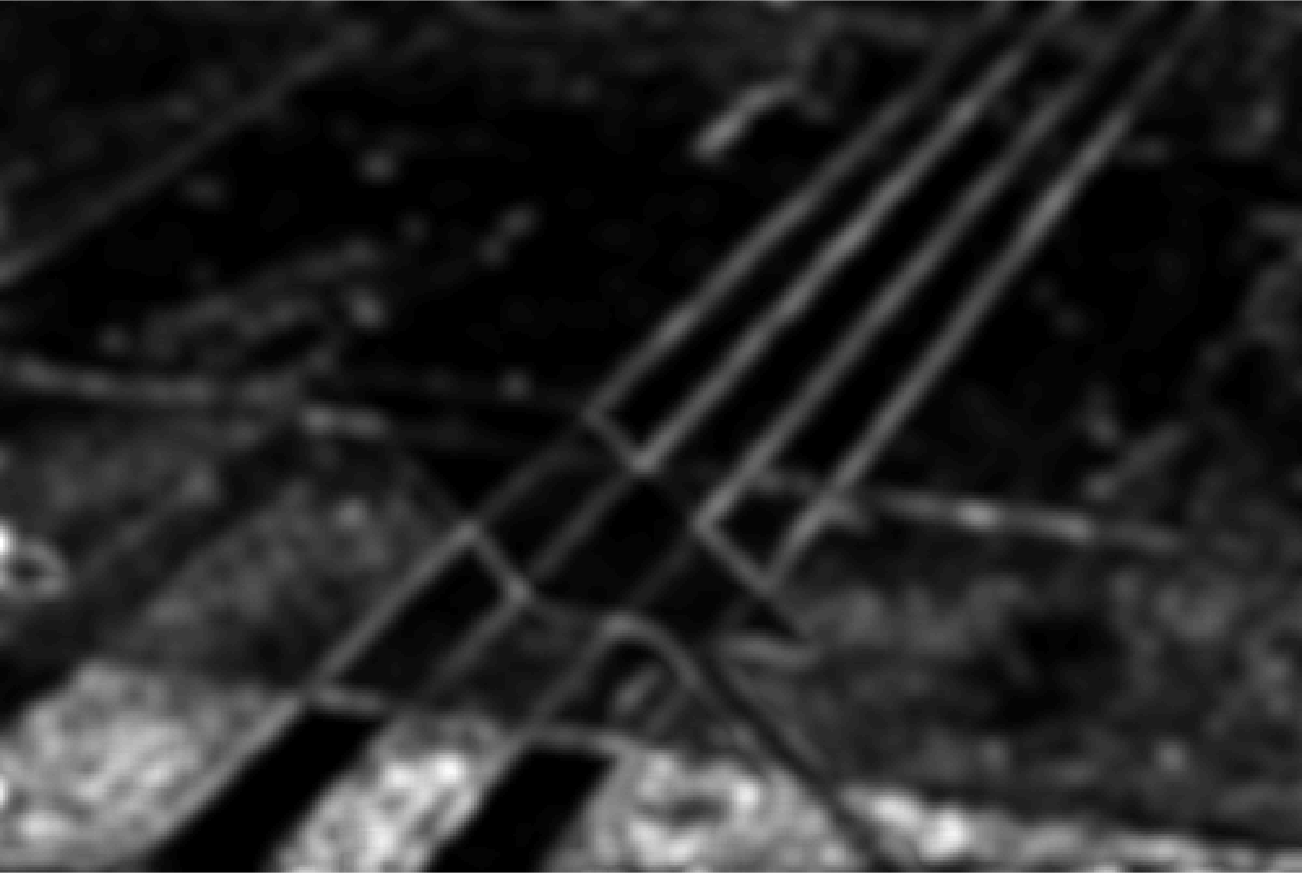}
\caption{STF, $\lambdabold_1$}
\label{fig: STVvsTVF2}
\end{subfigure}
\hfill
\begin{subfigure}[t]{0.193\textwidth}
\includegraphics[width=1\textwidth]{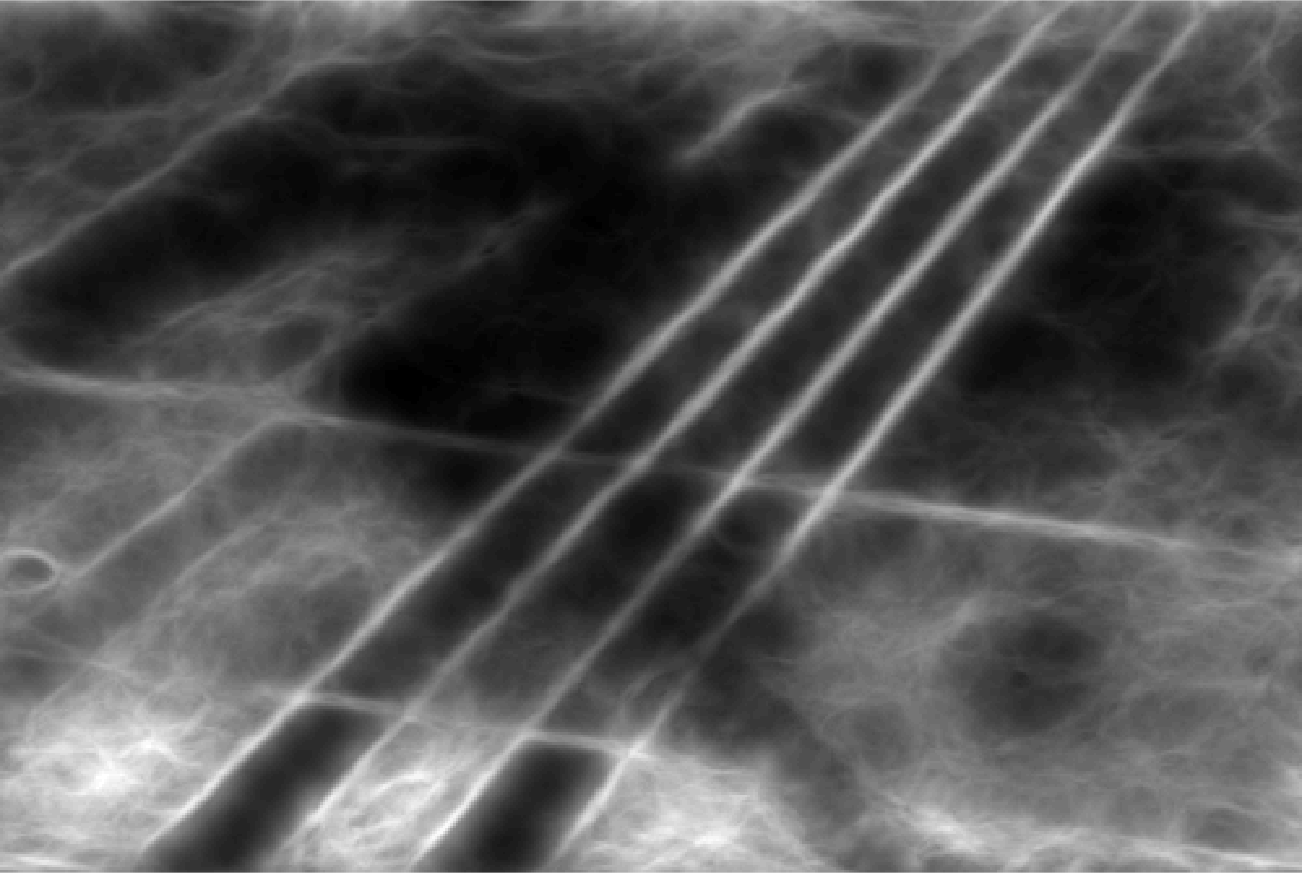}
\caption{TVF, $\lambdabold_1$}
\label{fig: STVvsTVF3}
\end{subfigure}
\hfill
\begin{subfigure}[t]{0.193\textwidth}
\includegraphics[width=1\textwidth]{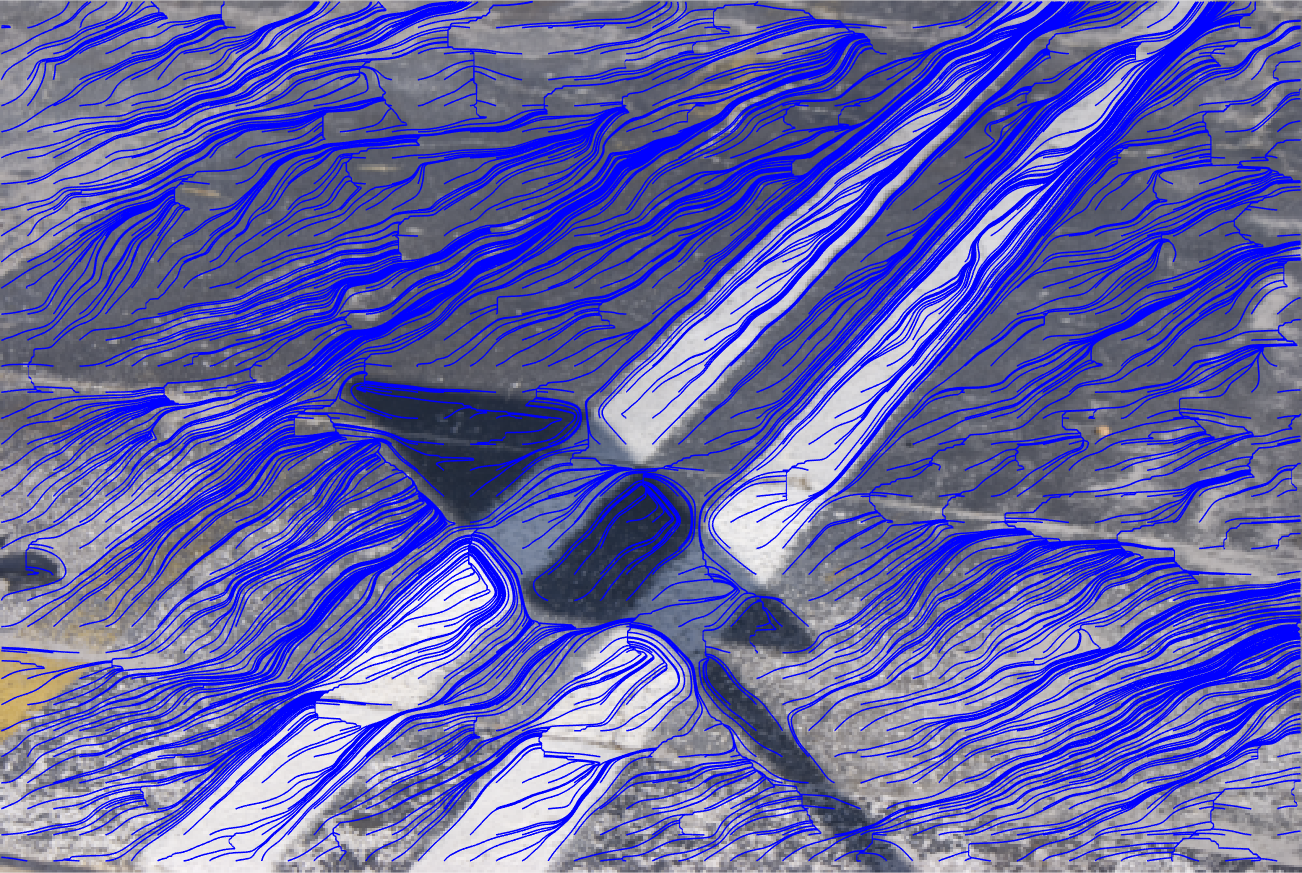}
\caption{STF, $\ebold_1$}
\label{fig: STVvsTVF4}
\end{subfigure}
\hfill\begin{subfigure}[t]{0.193\textwidth}
\includegraphics[width=1\textwidth]{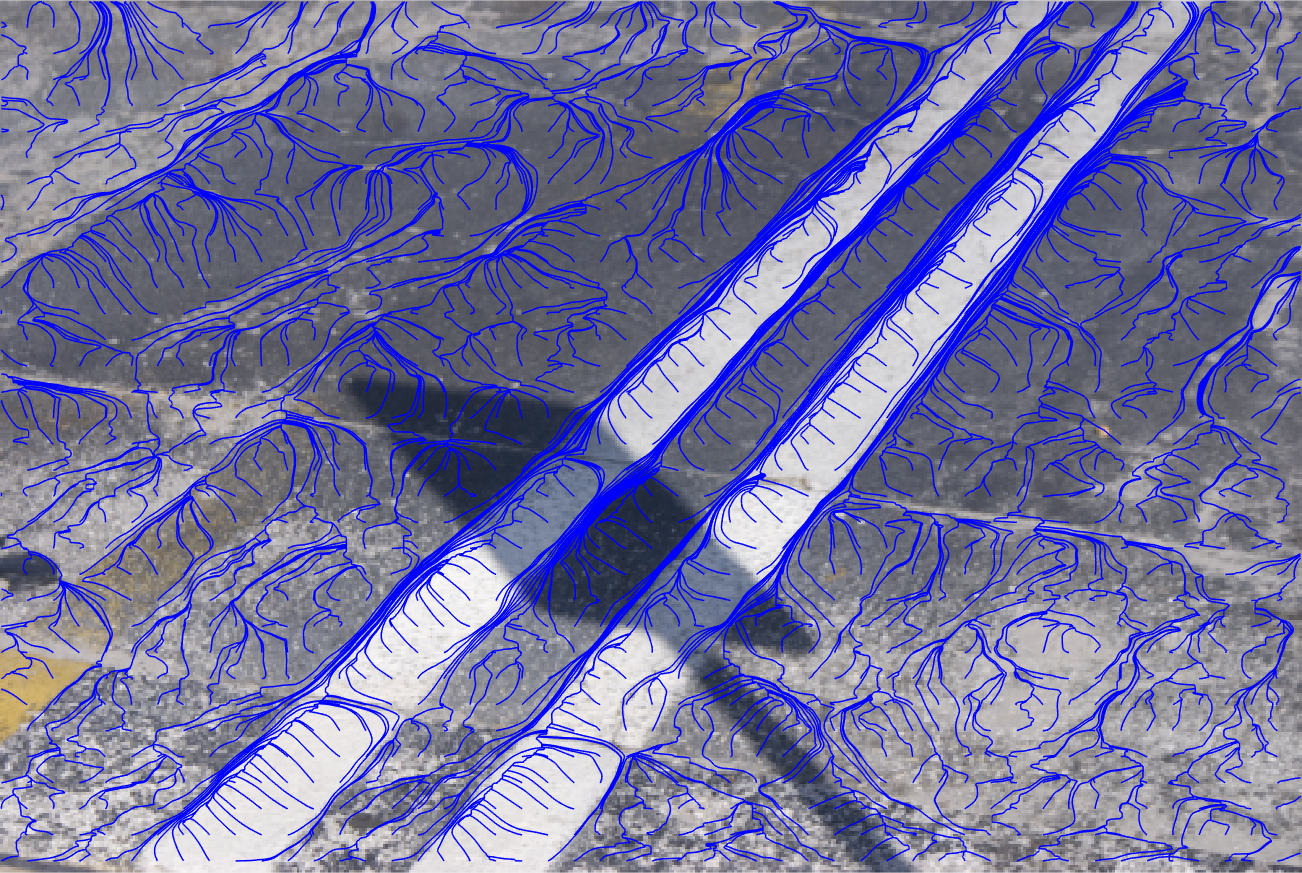}
\caption{TVF, $\ebold_1$}
\label{fig: STVvsTVF5}
\end{subfigure}
\caption{Comparison: structure tensor framework (STF) with $(\sigma,\rho)=(0.5,4)$ versus tensor voting framework (TVF) with multi-scales $(5,10,15)$ and $\sigma=0.5$. We plot the main direction $\ebold_1$ and its associated eigenvalue $\lambdabold_1$.}
\label{fig: STVvsTVF}
\end{figure}

\subsubsection{Results on real images.}\label{sec: real examples}
We now combine the proposed tensor voting framework with the anisotropic osmosis model to remove constant shadows by means of an anisotropic drift-diffusion model.
In the following experiments we use the final time $T=100000$ and the time step size $\tau=1000$. 
In practice the shadow removal is accomplished almost completely already for $t\ll T$. However, for a better approximation of the steady state, we use the final time $T$ for comparison.

In Figure \ref{fig: bamboo zoom} we show a zoom of the results from Figure \ref{fig:anisotropic_shadow}\footnote{Figure \ref{fig: full bamboo} (zoomed in Figure \ref{fig: bamboo}) courtesy of R. D. Kongskov.}, presented as motivation for this work.
Note that in this case the shadow is artificially added as a multiplicative rescaling factor $c\in(0,1)$. The estimation of the direction on the shadow boundary is computed via the proposed Algorithm \ref{alg: tensor voting}.
\begin{figure}[tbhp]
\centering
\begin{subfigure}[t]{0.225\textwidth}\centering
\includegraphics[width=1\textwidth,trim=2.5cm 3.5cm 3.5cm 2.5cm,clip=true]{\detokenize{./images/realshadow/44.png}}
\caption{Shadowed $\fbold$}
\label{fig: bamboo}
\end{subfigure}
\hfill
\begin{subfigure}[t]{0.225\textwidth}\centering
\includegraphics[width=1\textwidth,trim=2.5cm 3.5cm 3.5cm 2.5cm,clip=true]{\detokenize{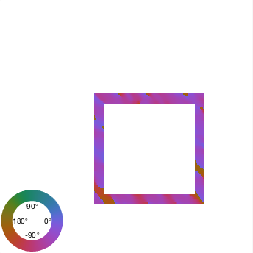}}
\caption*{Orientation angle $\thetabold$}
\end{subfigure}
\hfill
\begin{subfigure}[t]{0.225\textwidth}\centering
\includegraphics[width=1\textwidth,trim=2.5cm 3.5cm 3.5cm 2.5cm,clip=true]{\detokenize{./images/realshadow/44_classic.png}}
\caption*{Isotropic}
\end{subfigure}
\hfill
\begin{subfigure}[t]{0.225\textwidth}\centering
\includegraphics[width=1\textwidth,trim=2.5cm 3.5cm 3.5cm 2.5cm,clip=true]{\detokenize{./images/realshadow/44_mirebeau.png}}
\caption*{Anisotropic 
}
\end{subfigure}
\caption{Zoom of the results for the shadowed image in Figure \ref{fig:anisotropic_shadow}.}
\label{fig: bamboo zoom}
\end{figure}

In Figure \ref{fig: real images}  we apply the isotropic and the anisotropic model to a several shadowed images affected by natural constant shadows\footnote{
Figure \ref{fig: ball} from \url{http://www.cs.huji.ac.il/~danix/ShadowRemoval/index.html};  
Figure \ref{fig: signal} from \url{http://aqua.cs.uiuc.edu/site/projects/shadow.html};
Figure \ref{fig: people} from \url{http://www.cs.haifa.ac.il/hagit/papers/ShadowRemoval/};
}. Zoomed details can be found in Figure \ref{fig: real images zoom}.
Here the complexity of local image structures makes the application of the anisotropic osmosis model harder. Thus, we use the tensor voting algorithm \ref{alg: tensor voting} to estimate the local directions connecting the structures from inside to the outside of the shadow region. Also in these real-world scenarios, we observe the superiority of the anisotropic approach: structures are interpolated reliably across the light jump introduced by the shadow.

\begin{figure}[tbhp]
\centering
\begin{subfigure}[t]{0.225\textwidth}\centering
\includegraphics[width=1\textwidth]{\detokenize{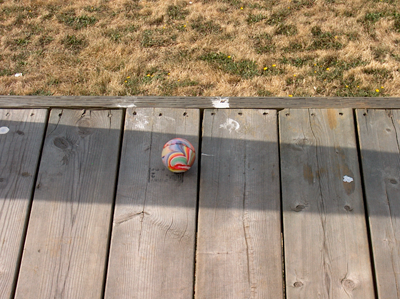}}
\caption{Shadowed $\fbold$}
\label{fig: ball}
\end{subfigure}
\hfill
\begin{subfigure}[t]{0.225\textwidth}\centering
\includegraphics[width=1\textwidth]{\detokenize{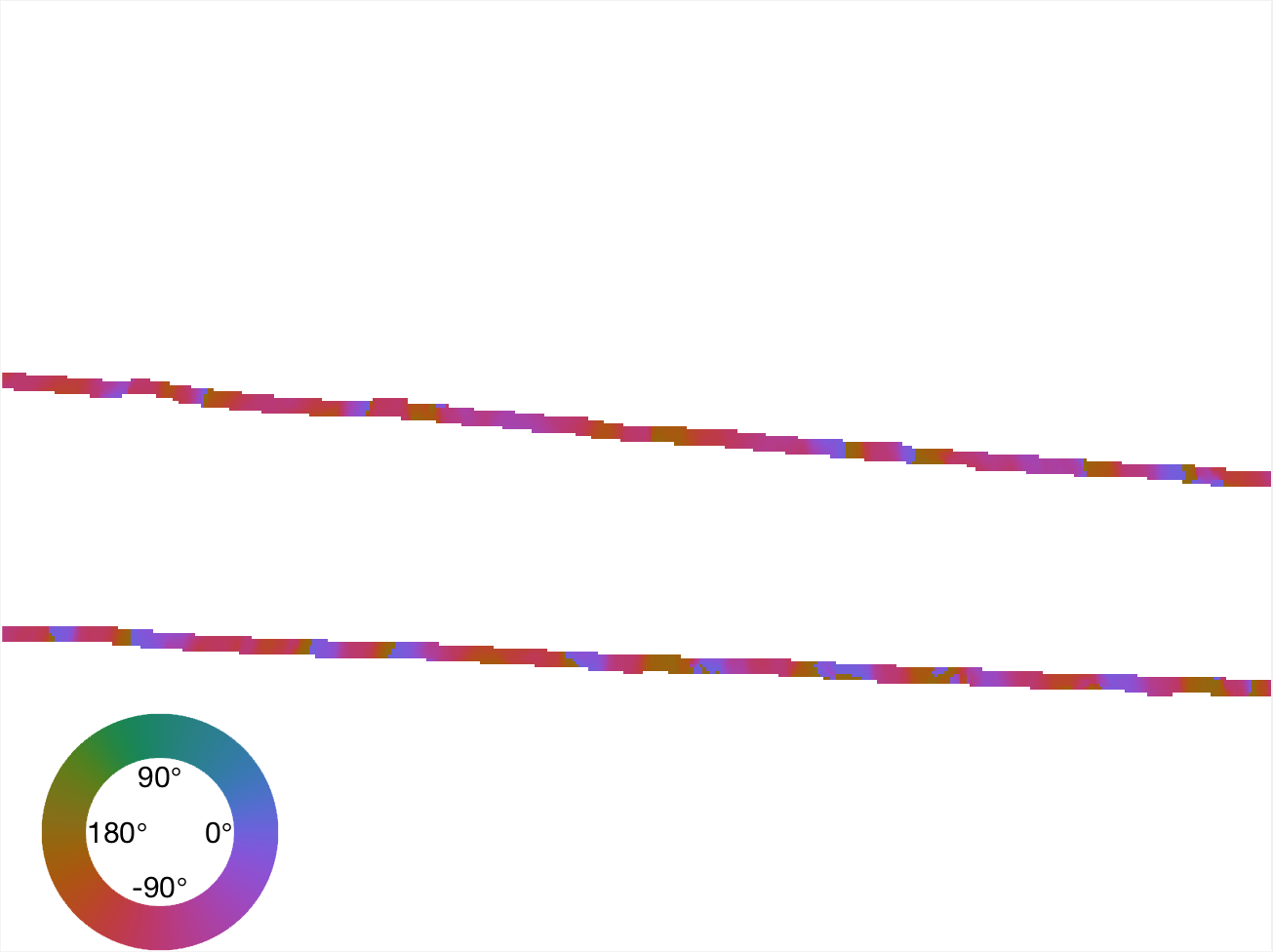}}
\caption*{Orientation angle $\thetabold$}
\end{subfigure}
\hfill
\begin{subfigure}[t]{0.225\textwidth}\centering
\includegraphics[width=1\textwidth]{\detokenize{./images/realshadow/6_classic.png}}
\caption*{Isotropic}
\end{subfigure}
\hfill
\begin{subfigure}[t]{0.225\textwidth}\centering
\includegraphics[width=1\textwidth]{\detokenize{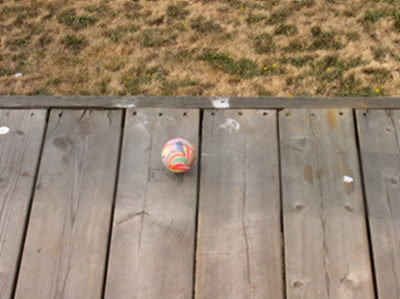}}
\caption*{Anisotropic}
\end{subfigure}
\\
\begin{subfigure}[t]{0.225\textwidth}\centering
\includegraphics[width=1\textwidth]{\detokenize{./images/realshadow/15.png}}
\caption{Shadowed $\fbold$}
\label{fig: signal}
\end{subfigure}
\hfill
\begin{subfigure}[t]{0.225\textwidth}\centering
\includegraphics[width=1\textwidth]{\detokenize{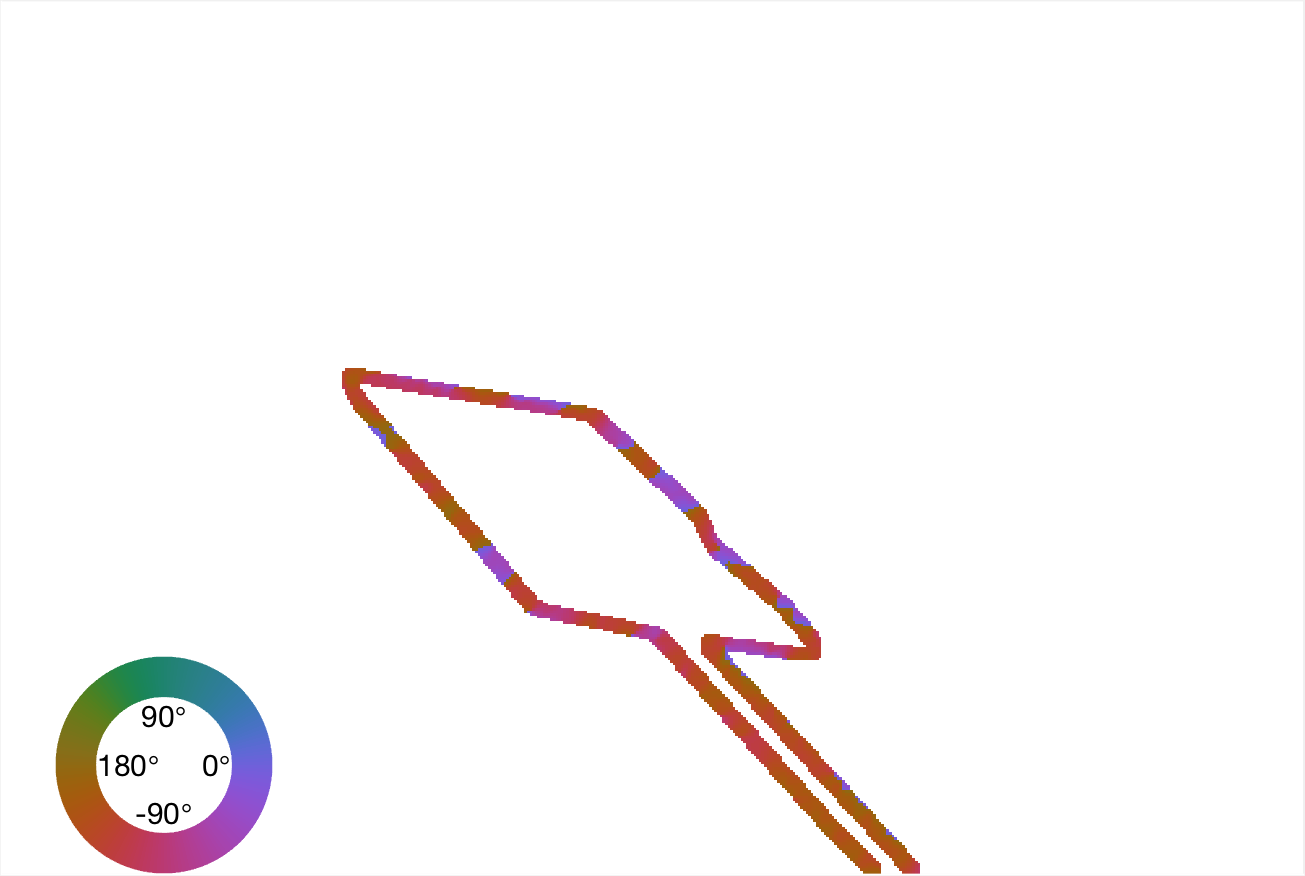}}
\caption*{Orientation angle $\thetabold$}
\end{subfigure}
\hfill
\begin{subfigure}[t]{0.225\textwidth}\centering
\includegraphics[width=1\textwidth]{\detokenize{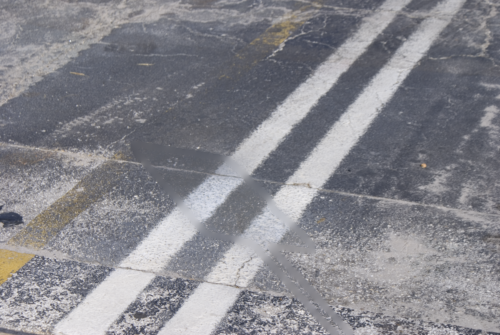}}
\caption*{Isotropic}
\end{subfigure}
\hfill
\begin{subfigure}[t]{0.225\textwidth}\centering
\includegraphics[width=1\textwidth]{\detokenize{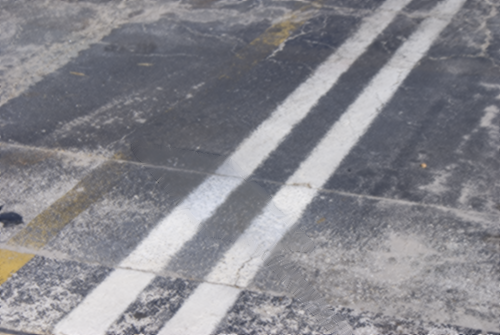}}
\caption*{Anisotropic}
\end{subfigure}
\\
\begin{subfigure}[t]{0.225\textwidth}\centering
\includegraphics[width=1\textwidth]{\detokenize{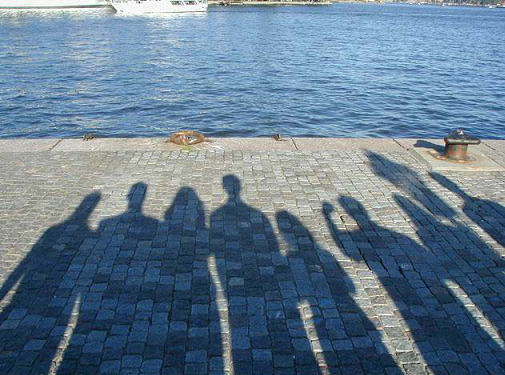}}
\caption{Shadowed $\fbold$}
\label{fig: people}
\end{subfigure}
\hfill
\begin{subfigure}[t]{0.225\textwidth}\centering
\includegraphics[width=1\textwidth]{\detokenize{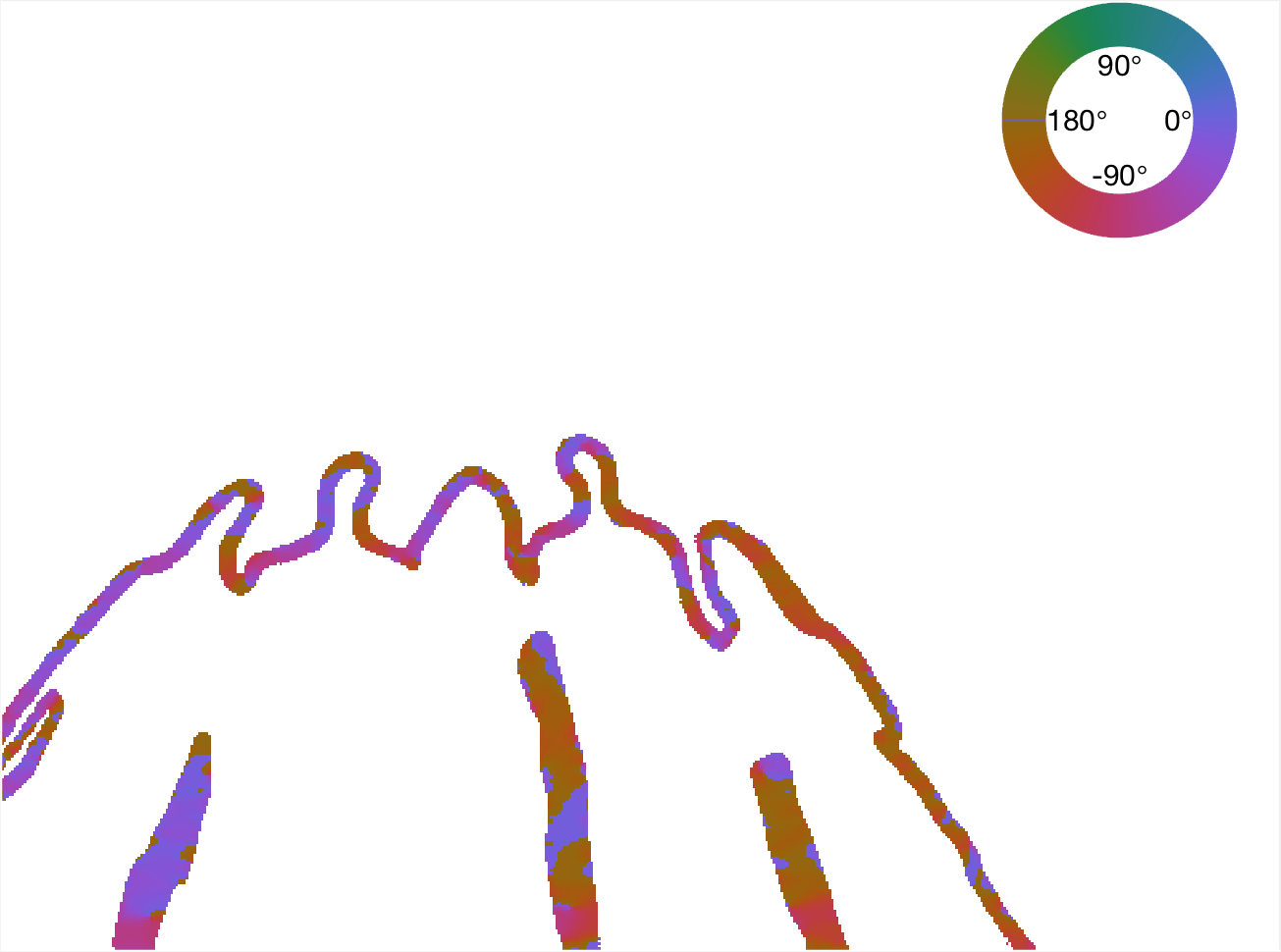}}
\caption*{Orientation angle $\thetabold$}
\end{subfigure}
\hfill
\begin{subfigure}[t]{0.225\textwidth}\centering
\includegraphics[width=1\textwidth]{\detokenize{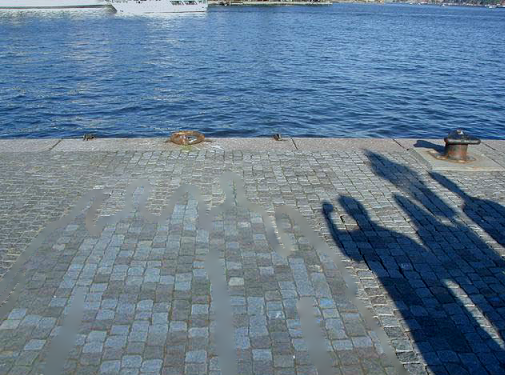}}
\caption*{Isotropic}
\end{subfigure}
\hfill
\begin{subfigure}[t]{0.225\textwidth}\centering
\includegraphics[width=1\textwidth]{\detokenize{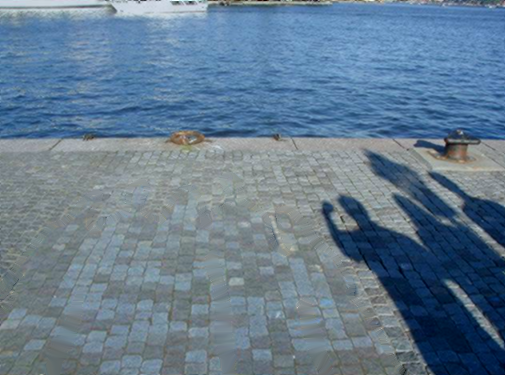}}
\caption*{Anisotropic}
\end{subfigure}
\caption{
Shadow removal via osmosis on real images. Comparison between isotropic and anisotropic osmosis.
Parameters: $\tau=1000$, $T=100000$, $\varepsilon=0.05$, and $\sigma=0.5$.
}
\label{fig: real images}
\end{figure}

\begin{figure}[tbhp]
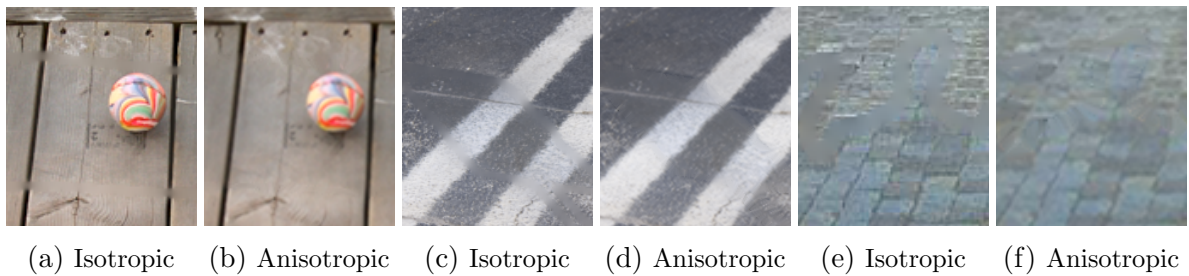

\begin{subfigure}[t]{0.161\textwidth}\centering
\includegraphics[width=1\textwidth, trim=3.55cm 2.5cm 6.6cm 3.5cm, clip=true]{\detokenize{./images/realshadow/6_classic.png}}
\caption{\footnotesize Isotropic}
\end{subfigure}
\begin{subfigure}[t]{0.161\textwidth}\centering
\includegraphics[width=1\textwidth, trim=3.55cm 2.5cm 6.6cm 3.5cm, clip=true]{\detokenize{./images/realshadow/6_mirebeau.png}}
\caption{\footnotesize Anisotropic}
\end{subfigure}
\begin{subfigure}[t]{0.161\textwidth}\centering
\includegraphics[width=1\textwidth, trim=5.05cm 2cm 7cm 3.4cm, clip=true]{\detokenize{./images/realshadow/15_classic.png}}
\caption{\footnotesize Isotropic}
\end{subfigure}
\begin{subfigure}[t]{0.161\textwidth}\centering
\includegraphics[width=1\textwidth, trim=5.05cm 2cm 7cm 3.4cm, clip=true]{\detokenize{./images/realshadow/15_mirebeau.png}}
\caption{\footnotesize Anisotropic}
\end{subfigure}
\begin{subfigure}[t]{0.161\textwidth}\centering
\includegraphics[width=1\textwidth, trim=3.01cm 4cm 12cm 6cm, clip=true]{\detokenize{./images/realshadow/14_classic.png}}
\caption{\footnotesize Isotropic}
\end{subfigure}
\begin{subfigure}[t]{0.161\textwidth}\centering
\includegraphics[width=1\textwidth, trim=3.01cm 4cm 12cm 6cm, clip=true]{\detokenize{./images/realshadow/14_mirebeau.png}}
\caption{\footnotesize Anisotropic}
\label{fig: zoom c}
\end{subfigure}
\caption{
Zoom into the results from Figure \ref{fig: real images}.
}
\label{fig: real images zoom}
\end{figure}

\subsection{Comments on the results}
We note that our anisotropic approach removes shadows effectively both for synthetic and real-world images. Structures are propagated correctly over the shadow boundary.

Since the AD-LBR stencil requires a positive definite matrix $\Wbold$ as input, we need to choose 
$\varepsilon > 0$, which may lead to some over-smoothing in the orthogonal direction. Also, we noticed some over-smoothing effect in $\Omega_\text{out}$ and $\Omega_\text{in}$ for the AD-LBR scheme, e.g.~ in Figure \ref{fig: zoom c}, whose understanding is a matter of future research.

In terms of efficiency, the computational time needed for computing the action of the exponential matrix onto a vector, that is the product $\exp(\tau\Abold)\ubold^0$, largely depends on the choice of the time-step $\tau$, the final time $T$ and the size of the images. Although it is possible to directly choose $\tau=T$ and proceed in a single
step, we prefer multiple steps. In this way, by comparing two successive solutions $\ubold^{k+1}$ and $\ubold^k$, it is possible to detect whether the evolution is sufficiently close to its steady state.
Our solvers that are exact in time offer additional advantages over classical inexact methods such as explicit and implicit schemes when one is also interested in good approximations of intermediate results. 



\section{Conclusions and outlook}

In this work, we have generalised isotropic osmosis filtering introduced in \cite{vogel,weickert} to the anisotropic setting. This was achieved by introducing a weight matrix whose directional information was extracted from a
modified tensor voting approach \cite{MorBurWeiGarPui12}. 
When applied to the shadow removal problem, the anisotropic model acts as an inpainting interpolator on the shadow boundary. It is close in spirit to inpainting with edge-enhancing anisotropic diffusion \cite{WW06,GWWB08,SPME14}.

From a numerical point of view,
we have combined the nonnegativity preserving anisotropic diffusion stencil of Fehrenbach and Mirebeau \cite{Mirebeau} with techniques based on exponential integration \cite{Caliari16}: We argued that the fully discrete model satisfies the discrete properties studied in \cite{vogel} in a general setting for osmosis. We tested the proposed model for synthetic and real examples, showing that the generalised model acts as a combined osmosis-inpainting model for shadow removal problems, thus avoiding any undesirable post-processing inpainting step.

Future work will address the investigation on the applicability of the proposed anisotropic model to more general imaging applications.

\ack{
SP acknowledges UK EPSRC grant EP/L016516/1 for the University of Cambridge, Cambridge Centre for Analysis DTC.
LC acknowledges the Fondation Mathématique Jacques Hadamard (FMJH), the JCJC INS2I grant LiftMe funded by CNRS and the RISE EU project NoMADS.
CBS acknowledges support from Leverhulme Trust project on Breaking the non-convexity barrier, EPSRC grant Nr.\ EP/M00483X/1, the EPSRC Centre Nr.\ EP/N014588/1, the RISE projects CHiPS and NoMADS, the Cantab Capital Institute for the Mathematics of Information and the Alan Turing Institute. 
JW acknowledges partially funding through the ERC Advanced Grant INCOVID.
The authors thank J.M.\ Mirebeau for the insightful discussions and comments on the adaptation of the code from \cite{Mirebeau} to our problem and the Isaac Newton Institute for Mathematical Sciences for support and hospitality during the program “Variational Methods and Effective Algorithms for Imaging and Vision” when work on this paper was undertaken. This work was supported by EPSRC grant number EP/K032208/1.
}

\section*{References}
\bibliographystyle{spmpsci}
\bibliography{biblio}

\begin{thebibliography}{10}
\providecommand{\url}[1]{{#1}}
\providecommand{\urlprefix}{URL }
\expandafter\ifx\csname urlstyle\endcsname\relax
  \providecommand{\doi}[1]{DOI~\discretionary{}{}{}#1}\else
  \providecommand{\doi}{DOI~\discretionary{}{}{}\begingroup
  \urlstyle{rm}\Url}\fi

\bibitem{AlMohy11}
Al-Mohy, A.H., Higham, N.J.: Computing the action of the matrix exponential
  with an application to exponential integrators.
\newblock SIAM Journal on Scientific Computing \textbf{33}(2), 488--511 (2011).
\newblock \doi{10.1137/100788860}

\bibitem{Arias2011}
Arias, P., Facciolo, G., Caselles, V., Sapiro, G.: A variational framework for
  exemplar-based image inpainting.
\newblock International Journal of Computer Vision \textbf{93}(3), 319--347
  (2011).
\newblock \doi{10.1007/s11263-010-0418-7}

\bibitem{aubert2006mathematical}
Aubert, G., Kornprobst, P.: Mathematical problems in image processing:
  {P}artial {D}ifferential {E}quations and the {C}alculus of {V}ariations, vol.
  147.
\newblock Springer-Verlag, New York (2006).
\newblock \doi{10.1007/978-0-387-44588-5}

\bibitem{BabaSiggraph2003}
Baba, M., Asada, N.: Shadow removal from a real picture.
\newblock In: ACM SIGGRAPH 2003 Sketches \& Applications, SIGGRAPH '03, pp.
  1--1. ACM, New York, NY, USA (2003).
\newblock \doi{10.1145/965400.965488}

\bibitem{Calatroni2017}
Calatroni, L., Estatico, C., Garibaldi, N., Parisotto, S.: Alternating
  direction implicit ({ADI}) schemes for a {PDE}-based image osmosis model.
\newblock Journal of Physics: Conference Series \textbf{904}(1), 012,014
  (2017).
\newblock \doi{10.1088/1742-6596/904/1/012014}

\bibitem{Caliari16}
Caliari, M., Kandolf, P., Ostermann, A., Rainer, S.: The {L}eja method
  revisited: Backward error analysis for the matrix exponential.
\newblock SIAM Journal on Scientific Computing \textbf{38}(3), A1639--A1661
  (2016).
\newblock \doi{10.1137/15M1027620}

\bibitem{ChanShen}
Chan, T., Shen, J.: Image Processing and Analysis: Variational, PDE, Wavelet,
  and Stochastic Methods.
\newblock Society for Industrial and Applied Mathematics, Philadelphia (2005).
\newblock \doi{10.1137/1.9780898717877}

\bibitem{ShadowReview2013}
Chunxia, X., Ruiyun, S., Donglin, X., Kwanâ-Liu, M.: Fast shadow removal using
  adaptive multi-scale illumination transfer.
\newblock Computer Graphics Forum \textbf{32}(8), 207--218 (2013).
\newblock \doi{10.1111/cgf.12198}

\bibitem{ConSlo}
Conway, J.H., Sloane, N.J.A.: Low-dimensional lattices. vi. {V}oronoi reduction
  of three-dimensional lattices.
\newblock Proceedings of the Royal Society of London A: Mathematical, Physical
  and Engineering Sciences \textbf{436}(1896), 55--68 (1992).
\newblock \doi{10.1098/rspa.1992.0004}

\bibitem{Mirebeau}
Fehrenbach, J., Mirebeau, J.M.: Sparse non-negative stencils for anisotropic
  diffusion.
\newblock Journal of Mathematical Imaging and Vision \textbf{49}(1), 123--147
  (2014).
\newblock \doi{10.1007/s10851-013-0446-3}

\bibitem{Forstner86}
F\"{o}rstner, W.: {A Feature Based Correspondence Algorithm for Image
  Matching}.
\newblock Int. Arch. of Photogrammetry and Remote Sensing \textbf{26}(3),
  150--166 (1986)

\bibitem{FraAlmRonFloRom06}
Franken, E., van Almsick, M., Rongen, P., Florack, L., ter Haar~Romeny, B.: An
  efficient method for tensor voting using steerable filters.
\newblock In: A.~Leonardis, H.~Bischof, A.~Pinz (eds.) Computer Vision - ECCV
  2006, pp. 228--240. Springer Berlin Heidelberg, Berlin, Heidelberg (2006).
\newblock \doi{10.1007/11744085_18}

\bibitem{GWWB08}
Gali\'c, I., Weickert, J., Welk, M., Bruhn, A., Belyaev, A., Seidel, H.P.:
  Image compression with anisotropic diffusion.
\newblock Journal of Mathematical Imaging and Vision \textbf{31}(2--3),
  255--269 (2008).
\newblock \doi{10.1007/s10851-008-0087-0}

\bibitem{GL10}
Grasmair, M., Lenzen, F.: Anisotropic total variation filtering.
\newblock Applied Mathematics and Optmization \textbf{62}, 323--339 (2010).
\newblock \doi{10.1007/s00245-010-9105-x}

\bibitem{Guichard2001}
Guichard, F., Moisan, L., Morel, J.M.: A review of {PDE} models in image
  processing and image analysis.
\newblock Journal de Physique IV pp. 137--154 (2002).
\newblock \doi{10.1051/jp42002006}

\bibitem{GidMed96}
Guy, G., Medioni, G.: Inferring global perceptual contours from local features.
\newblock International Journal of Computer Vision \textbf{20}(1), 113--133
  (1996).
\newblock \doi{10.1007/BF00144119}

\bibitem{Hagenburg2012}
Hagenburg, K., Breu{\ss}, M., Weickert, J., Vogel, O.: Novel schemes for
  hyperbolic pdes using osmosis filters from visual computing.
\newblock In: A.M. Bruckstein, B.M. ter Haar~Romeny, A.M. Bronstein, M.M.
  Bronstein (eds.) Scale Space and Variational Methods in Computer Vision, pp.
  532--543. Springer, Berlin (2012).
\newblock \doi{10.1007/978-3-642-24785-9_45}

\bibitem{HarrisStephens88}
Harris, C., Stephens, M.: A combined corner and edge detector.
\newblock In: Proceedings of the Alvey Vision Conference, pp. 23.1--23.6.
  Alvety Vision Club (1988).
\newblock \doi{10.5244/C.2.23}

\bibitem{HochOst10}
Hochbruck, M., Ostermann, A.: Exponential integrators.
\newblock Acta Numerica \textbf{19}, 209–286 (2010).
\newblock \doi{10.1017/S0962492910000048}

\bibitem{KassWitkin85}
Kass, M., Witkin, A.: Analyzing oriented patterns.
\newblock Computer Vision, Graphics, and Image Processing \textbf{37}(3), 362
  -- 385 (1987).
\newblock \doi{10.1016/0734-189X(87)90043-0}

\bibitem{KonDon17}
Kongskov, R.D., Dong, Y.: Directional total generalized variation
  regularization for impulse noise removal.
\newblock In: F.~Lauze, Y.~Dong, A.B. Dahl (eds.) Scale Space and Variational
  Methods in Computer Vision, pp. 221--231. Springer International Publishing,
  Cham (2017).
\newblock \doi{10.1007/978-3-319-58771-4_18}

\bibitem{DongDTGV2017}
Kongskov, R.D., Dong, Y., Knudsen, K.: Directional total generalized variation
  regularization  (2017).
\newblock ArXiv preprint: 1701.02675

\bibitem{MagMenFre14}
Maggiori, E., Manterola, H.L., del Fresno, M.: Perceptual grouping by tensor
  voting: a comparative survey of recent approaches.
\newblock IET Computer Vision \textbf{9}(2), 259--277 (2015).
\newblock \doi{10.1049/iet-cvi.2014.0103}

\bibitem{MorBurWeiGarPui12}
Moreno, R., Pizarro, L., Burgeth, B., Weickert, J., Garcia, M.A., Puig, D.:
  Adaptation of tensor voting to image structure estimation.
\newblock In: D.H. Laidlaw, A.~Vilanova (eds.) New Developments in the
  Visualization and Processing of Tensor Fields, pp. 29--50. Springer (2012)

\bibitem{MN01a}
Mr\'azek, P., Navara, M.: Consistent positive directional splitting of
  anisotropic diffusion.
\newblock In: B.~Likar (ed.) Proc.~Sixth Computer Vision Winter Workshop, pp.
  37--48. Bled, Slovenia (2001)

\bibitem{NagelEnkelmann1986}
Nagel, H.H., Enkelmann, W.: An investigation of smoothness constraints for the
  estimation of displacement vector fields from image sequences.
\newblock IEEE Trans. Pattern Anal. Mach. Intell. \textbf{8}(5), 565--593
  (1986).
\newblock \doi{10.1109/TPAMI.1986.4767833}

\bibitem{ParCalDaf2018}
Parisotto, S., Calatroni, L., Daffara, C.: Digital cultural heritage imaging
  via osmosis filtering.
\newblock In: A.~Mansouri, A.~El~Moataz, F.~Nouboud, D.~Mammass (eds.) Image
  and Signal Processing LNCS 10884, pp. 407--415. Springer International
  Publishing, Cham (2018).
\newblock \doi{10.1007/978-3-319-94211-7_44}

\bibitem{ParMasSch18applied}
Parisotto, S., Masnou, S., Sch{\"o}nlieb, C.B.: Higher order total directional
  variation. {P}art {I}: Imaging applications  (forthcoming)

\bibitem{ParMasSch18analysis}
Parisotto, S., Masnou, S., Sch{\"o}nlieb, C.B.: Higher order total directional
  variation. {P}art {II}: Analysis  (forthcoming)

\bibitem{ProtterWeinberger}
Protter, M.H., Weinberger, H.F.: Maximum Principles in Differential Equations.
\newblock Springer-Verlag New York (1984).
\newblock \doi{10.1007/978-1-4612-5282-5}

\bibitem{Saad92}
Saad, Y.: Analysis of some {K}rylov subspace approximations to the matrix
  exponential operator.
\newblock SIAM Journal on Numerical Analysis \textbf{29} (1992).
\newblock \doi{10.1137/0729014}

\bibitem{Sa01}
Sapiro, G.: Geometric Partial Differential Equations and Image Analysis.
\newblock Cambridge University Press, Cambridge, UK (2001).
\newblock \doi{10.1017/CBO9780511626319}

\bibitem{SPME14}
Schmaltz, C., Peter, P., Mainberger, M., Ebel, F., Weickert, J., Bruhn, A.:
  Understanding, optimising, and extending data compression with anisotropic
  diffusion.
\newblock International Journal of Computer Vision \textbf{108}(3), 222--240
  (2014).
\newblock \doi{10.1007/s11263-014-0702-z}

\bibitem{schoenlieb}
Schönlieb, C.B.: Partial Differential Equation Methods for Image Inpainting.
\newblock Cambridge Monographs on Applied and Computational Mathematics.
  Cambridge University Press (2015).
\newblock \doi{10.1017/CBO9780511734304}

\bibitem{Tarjan1972}
Tarjan, R.: Depth-first search and linear graph algorithms (1972).
\newblock \doi{10.1137/0201010}

\bibitem{vogel}
Vogel, O., Hagenburg, K., Weickert, J., Setzer, S.: A fully discrete theory for
  linear osmosis filtering.
\newblock In: A.~Kuijper, K.~Bredies, T.~Pock, H.~Bischof (eds.) Scale Space
  and Variational Methods in Computer Vision, pp. 368--379. Springer Berlin
  Heidelberg, Berlin, Heidelberg (2013).
\newblock \doi{10.1007/978-3-642-38267-3_31}

\bibitem{weickert98}
Weickert, J.: {Anisotropic Diffusion in Image Processing}.
\newblock B.G. Teubner, Stuttgart (1998)

\bibitem{weickert}
Weickert, J., Hagenburg, K., Breu{\ss}, M., Vogel, O.: Linear osmosis models
  for visual computing.
\newblock In: A.~Heyden, F.~Kahl, C.~Olsson, M.~Oskarsson, X.C. Tai (eds.)
  Energy Minimization Methods in Computer Vision and Pattern Recognition, pp.
  26--39. Springer Berlin Heidelberg, Berlin, Heidelberg (2013).
\newblock \doi{10.1007/978-3-642-40395-8_3}

\bibitem{WW06}
Weickert, J., Welk, M.: Tensor field interpolation with {PDEs}.
\newblock In: J.~Weickert, H.~Hagen (eds.) Visualization and Processing of
  Tensor Fields, pp. 315--325. Springer, Berlin (2006).
\newblock \doi{10.1007/3-540-31272-2_19}

\bibitem{weickert2013b}
Weickert, J., Welk, M., Wickert, M.: L2-stable nonstandard finite differences
  for anisotropic diffusion.
\newblock In: A.~Kuijper, K.~Bredies, T.~Pock, H.~Bischof (eds.) Scale Space
  and Variational Methods in Computer Vision, pp. 380--391. Springer Berlin
  Heidelberg, Berlin, Heidelberg (2013).
\newblock \doi{10.1007/978-3-540-72823-8}

\end{thebibliography}

\end{document}